\documentclass{amsart}
\usepackage{amsbsy}
\usepackage{mathrsfs}
\usepackage{amsmath}
\usepackage{amsfonts}
\usepackage{amssymb}
\usepackage{graphicx}%
\newtheorem{theorem}{Theorem}[section]
\newtheorem{lemma}[theorem]{Lemma}
\newtheorem{corollary}[theorem]{Corollary}
\newtheorem{proposition}[theorem]{Proposition}

\theoremstyle{definition}
\newtheorem{definition}[theorem]{Definition}

\theoremstyle{remark}
\newtheorem{remark}[theorem]{Remark}  
\numberwithin{equation}{section}

\usepackage{color}
\usepackage{fancyhdr}
\allowdisplaybreaks[4]

\begin{document}

	\title{On Talagrand's functional and generic chaining}

	
	\author{Yiming Chen}
	\address{Institute for Financial Studies, Shandong University,  Jinan,  250100, China.}
	\email{chenyiming960212@mail.sdu.edu.cn}

	\author{Pengtao Li}
	\address{Department of Mathematics, University of Southern California, Los Angeles, CA 90089, United States.}
	\email{pengtaol@usc.edu}

	\author{Dali Liu}
	\address{Department of Statistics and Probability, Michigan State University, East Lansing, MI 48824, United States.}
	\email{liudali@msu.edu}

	\author{Hanchao Wang }
	\address{Institute for Financial Studies, Shandong University,  Jinan,  250100, China.}
	\email{wanghanchao@sdu.edu.cn}
	

	\subjclass[2020]{60E15;  60G17;  60G50}

	\date{}

	\keywords{$\varphi$-sub-Gaussian distribution; Generic chaining; Tail probability inequality;  Compressed sensing; Johnson-Lindenstrauss lemma; Order 2 Gaussian chaos. }

	\begin{abstract}
		In the study of the supremum of stochastic processes, Talagrand's chaining functionals and his generic chaining method are heavily related to the distribution of stochastic processes. In the present paper, we construct Talagrand's type functionals in the general distribution case and obtain the upper bound for the suprema of all $p$-th moments of the stochastic process using the generic chaining method. As applications, we obtained the Johnson-Lindenstrauss lemma, the upper bound for the supremum of all $p$-th moment of order 2 Gaussian chaos, and convex signal recovery in our setting.
		
	\end{abstract}

	\maketitle

	\section{Introduction}\label{sec-Intro}
	
	A family of random variables $(X_t)_{t\in T}$ is referred to as a stochastic process, where $T$ is the index set. The main objective of this paper is to investigate
	$$\left(\textsf{E}\sup_{t\in T}|X_t|^p\right)^{1/p}=:\sup\left\{\left(\textsf{E}\sup_{t\in F}|X_t|^p\right)^{1/p}, F\subset T, F ~\text{is finite}\right\},$$
	where $p\ge 1$. Prior to this paper, several authors have discussed $\mathsf{E}\sup_{t\in T}X_t$, and outstanding works have been done by Dudley \cite{RD}, Fernique \cite{XF}, Talagrand \cite{t0,t1,MT}, van Handel \cite{RV}, among others. It is worth mentioning that Talagrand developed a comprehensive framework, including Talagrand chaining functionals and generic chaining techniques, to handle the boundness of stochastic processes. Under the assumption that $(X_t)_{t\in T}$ is a Gaussian process, Talagrand \cite{t1} derived
	\begin{eqnarray}\label{maj}
		\frac{1}{L}\gamma_2(T,d)\le \mathsf{E}\sup_{t\in T}X_t\le L \gamma_2(T,d),
	\end{eqnarray}
	where $d$ is a distance defined on $T$, $\gamma_2(T,d)$ is the so-called Talagrand's chaining functional, and $L$ is a constant.  Result (\ref{maj}) is also known as the majorizing measure theorem.
	
	In fact, it can be observed that the proof of the upper bound in (\ref{maj}) heavily relies on:
	\begin{equation}
		\label{sub}
		\mathsf{P}(|X_t-X_s|\ge u)\le 2\exp\left(-\frac{u^2}{2d^2(s,t)}\right),
	\end{equation}
	where $u>0$. Furthermore, for $u>0$, if $(X_t)_{t\in T}$ satisfies the following condition:
	\begin{equation}
		\label{b}
		\mathsf{P}(|X_t-X_s|\ge u)\le 2\exp\left(-\min\left\{\frac{u^2}{d_2^2(s,t)}, \frac{u}{d_1(s,t)}\right\}\right),
	\end{equation}
	with $d_1$ and $d_2$ two distances defined on $T$,
	Talagrand \cite{t2} proved that
	\begin{eqnarray}\label{bj}
		\mathsf{E}\sup_{t\in T}X_t\le L_1 \gamma_2(T,d_2)+ L_2 \gamma_1(T,d_1),
	\end{eqnarray}
	where $L_1$, $L_2$ are constants, and $\gamma_1(T,d_1)$, $\gamma_2(T,d_2)$ are Talagrand chaining functionals.
	
	(\ref{maj}) and (\ref{bj}) inspired us to consider that Talagrand's chaining functionals may depend on the distribution of $X$. The main contribution of this article is to construct distribution-dependent Talagrand-type chaining functionals under the assumption that $(X_t)_{t\in T}$ satisfies a general distribution condition. Using generic chaining techniques, we obtain an upper bound for $(\mathsf{E}\sup_{t\in T}|X_t|^p)^{1/p}$. Our results indicate that the Talagrand-type chaining functionals depend not only on the geometric of  $T$ but also heavily on the distributions of $(X_t)_{t\in T}$. We establish a series of theoretical properties for the constructed Talagrand-type chaining functionals and derive upper bounds for $(\mathsf{E}\sup_{t\in T}|X_t|^p)^{1/p}$ in various scenarios. It should be noted that studying the upper bound of $(\mathsf{E}\sup_{t\in T}|X_t|^p)^{1/p}$ has its own theoretical significance and can further lead to tail probability inequalities for $\sup_{t\in T}X_t$.                                       
	
	This paper considers the case where $X_t$ obeys $\varphi$-sub-Gaussian distributions. $\varphi$-sub-Gaussian distribution is introduced by Kozachenko and Ostrovsky \cite{KO1985}, as a generation of sub-Gaussian distribution.  As an essential part of high-dimensional probability theory, the properties and various applications (such as random matrix theory, stochastic optimization, and statistical learning theory) of sub-Gaussian distribution have been well studied. Some details of sub-Gaussian distribution can be found in Vershynin \cite{vr}. Dirksen extended the sub-Gaussian case to some special cases in   \cite{SD}, and also discussed the upper bound of $(\textsf{E}\sup_{t\in T}|X_t|^p)^{1/p}$. Compared to Dirksen \cite{SD},  our results are more general.
	
	We write $a\lesssim b$ for convenience if a universal constant $C$ satisfies $a\le Cb$. Additionally, we use the  symbol $a\asymp b$ if $a\lesssim b$ and $b\lesssim a$. The structure of the paper is as follows.   The preliminaries on $\varphi$-sub-Gaussian distributions are given in  Section 2. Section 3 introduces distribution-dependent Talagrand-type chaining functionals and presents their theoretical properties. The main results and proofs are presented in Section  4. Section  5 provides some examples and applications of our main results.  Some discussions are presented in Section 6.

	\section{$\varphi$-sub-Gaussian distribution}
	
	Let us begin with the definition of sub-Gaussian distribution.  For a random variable $\xi$ and $\lambda \in \mathbb{R}$, if 
	$$\textsf{E} \exp(\lambda \xi)\le  \exp\left(\lambda^2/2\right),$$
	we say that $\xi$ have sub-Gaussian distribution or $\xi$ is a sub-Gaussian random variable.  If $\{\xi_i\}_{1\le i\le n }$ are sub-Gaussian variables,  we can obtain
	\begin{eqnarray*}
		\textsf{E}\max_{1\le i\le n } \xi_i&\le & \inf_{\lambda>0} \frac{1}{\lambda}  \log \textsf{E}\exp(\lambda\max_{1\le i\le n } \xi_i)\\
		&\le & \inf_{\lambda>0} \frac{1}{\lambda}  \log \sum_{1\le i\le n}\textsf{E}\exp(\lambda \xi_i)\\
		&\le & \inf_{\lambda>0} \frac{1}{\lambda}  \left(\log n+\frac{\lambda^2}{2}\right)=\sqrt{2\log n},
	\end{eqnarray*}
	by Chernoff's approach.  The convexity of $\frac{x^2}{2}$ plays an essential role in this approach.  Kozachenko and Ostrovsky extended $\frac{x^2}{2}$ to more general cases in \cite{KO1985} and introduced the  $\varphi$-sub-Gaussian distribution.   $\varphi$-sub-Gaussian distribution is widely used in many application scenarios. Antonini and Hu \cite{AHKV} present an application to random Fourier series for $\varphi$-sub-Gaussian random variables. In the sampling theorem, Kozachenko and Olenko \cite{KO}  investigated $L_p([0, T])$ and uniform approximations of $\varphi$-sub-Gaussian random processes. Moreover, they also present the results on the Whittaker-Kotelnikov-Shannon (WKS) approximation of $\varphi$-sub-Gaussian signals in $L_p([0, T])$ with a given accuracy and reliability \cite{KO2016}. 

	This section introduces the definition and notations of $\varphi$-sub-Gaussian distribution. Some useful lemmas and properties are also collected in this section.  More details can be found in  Buldygin and Kozachenko  \cite{BK}.

	\begin{definition} A function $\varphi(x), x \in \mathbb{R}$, is called an Orlicz $N$-function if it is continuous, even, convex with $\varphi(0)=0$ and monotonically increasing in the set $x>0$, satisfying
		\begin{itemize}
			\item $\varphi(x)/x\rightarrow 0$, when $x\rightarrow0$; 
			\item $\varphi(x) / x \rightarrow \infty$, when $x \rightarrow \infty$.
		\end{itemize}
	\end{definition}
	
	\begin{definition} An Orlicz $N$-function $\varphi$ satisfies  $Q-$condition, if  there exist constant $c>0$,  such that
		$$
		\liminf _{x \rightarrow 0} \frac{\varphi(x)}{x^2}=c.
		$$
	\end{definition}

	\begin{remark}
		The following are some useful examples of Orlicz $N$-functions. \\
		1) $\varphi(x)=a|x|^\alpha, x \in \mathbb{R} ; a>0, \alpha>1$;\\
		2) $\varphi(x)=c\left(\exp \left\{|x|^\alpha\right\}-1\right), x \in \mathbb{R} ; c>0, \alpha>1$;\\
		3) $\varphi(x)=\exp \{|x|\}-|x|-1, x \in \mathbb{R}$.
	\end{remark}
	
	Now, we present two useful lemmas on  Orlicz $N$-functions, which can be found in Buldygin and Kozachenko \cite{BK}.
	
	\begin{lemma}\label{book}
		For any Orlicz $N$-function $\varphi$, the following results hold:
		\begin{itemize}
			\item for $\beta>1$, $\varphi(\beta x)\geq\beta\varphi(x)$ for $x\in\mathbb{R}$;
			\item there exists a constant $c=c(\varphi)>0$, such that $\varphi(x)>cx$ for $x>1$;
			\item the function $g(x)=\frac{\varphi(x)}{x}$ is monotone nondecreasing in $x$ for all $x>0$;
			\item $\varphi(x)+\varphi(y)\leq\varphi(|x|+|y|)$ for $x,y\in\mathbb{R}$.
		\end{itemize}
	\end{lemma}

	\begin{lemma}\label{inv}
		Denoting the inverse function of $\varphi(\cdot)$ by $\varphi^{(-1)}(\cdot)$, then for an Orlicz $N-$function $\varphi(\cdot)$, the following holds true:
		\begin{itemize}
			\item $\varphi^{(-1)}(\beta x)\leq\beta\varphi^{(-1)}(x),\quad x>0,$ for $\beta>1$;
			\item $\varphi^{(-1)}(\beta x)\geq\beta\varphi^{(-1)}(x),\quad x>0,$ for $\beta\leq1$.
		\end{itemize}
	\end{lemma}
	
	We need  {\em Young-Fenchel transform} for a given function $f$ to present our results.

	\begin{definition} Let $f=(f(x), x \in \mathbb{R})$ be a real-valued function. The function $f^*=\left(f^*(x), x \in \mathbb{R}\right)$ defined by the formula
		$$
		f^*(x)=\sup_{y \in \mathbb{R}}(x y-f(y)),
		$$
		is called the convex conjugate of $f$, also known as the  Young-Fenchel transform of the function $f$.
	\end{definition}
	
	\begin{remark}
		It is remarkable that the convex conjugate function of an Orlicz $N$-function is also an Orlicz $N$-function. 
	\end{remark}
	
	\begin{remark}
		Here are some examples of the Young-Fenchel transform for some functions.
		\begin{itemize}
			\item Given $a>1, c>0$, let $f(x)=c|x|^a, x \in \mathbb{R}$. Then
			$$
			f^*(x)=c_1|x|^b, \quad x \in \mathbb{R},
			$$
			where $c_1=(c a)^{b / a} / b$ and where $b$ is the conjugate exponent of $a$, namely $1 / b+1 / a=1$. 
			\item In particular, if $f(x)=\frac{1}{2} x^2, x \in \mathbb{R}$, then $f^{*}(x)=\frac{1}{2} x^2$.
			\item Let $f(x)=e^{|x|}-|x|-1, x \in \mathbb{R}$, then
			$
			f^*(x)=(|x|+1) \ln (|x|+1)-|x|, \quad x \in \mathbb{R}.
			$
		\end{itemize}
		\end{remark}

	Now, we present the definition of  $\varphi$-sub-Gaussian distribution.

	\begin{definition} 
		For an Orlicz $N$-function $\varphi$ satisfying $Q-$condition, a zero-mean random variable $\xi$ obeys $\varphi$-sub-Gaussian distribution, if there exists a constant $a\geq 0$ such that the inequality 
		\begin{equation}\label{defin}
			\textsf{E} \exp (\lambda \xi) \leq \exp \left(\varphi\left(\lambda a \right)\right),
		\end{equation}
		holds for all $\lambda \in \mathbb{R}$.
	\end{definition}

	\begin{remark}
		The assumptions that $\xi$ is zero-mean and that $\varphi$ satisfies   $Q-$Condition are actually necessary. As $\lambda\to0$, it is known that
		\begin{align*}
			\textsf{E}\exp{(\lambda\xi)}=1+\lambda\textsf{E}\xi+\frac{\lambda^2}{2}\textsf{E}\xi^2+o(\lambda^2),\\
			\exp{(\varphi(a\lambda))}=1+\varphi(a\lambda)+o(\varphi(a\lambda)).
		\end{align*}
		The inequality (\ref{defin}) implies these two assumptions.
	\end{remark}
	
	A random variable that obeys $\varphi$-sub-Gaussian distribution is also called a $\varphi$-sub-Gaussian random variable. Let $Sub_{\varphi}(\Omega)$ denotes the space of $\varphi$-sub-Gaussian random variables. For a metric space $T$, a stochastic  process, $(X_{t})_{t\in T}$, is called $\varphi$-sub-Gaussian process if $X_t \in Sub_{\varphi}(\Omega)$ for all $t \in T$. 
	
	\begin{lemma}{\textbf{(\cite{BK} Lemma 4.2)}}\label{l2}
		For any random variable $\xi \in \operatorname{Sub}_{\varphi}(\Omega)$, 
		\begin{equation}\label{2p}
			\textsf{E} \exp \{\lambda \xi\} \leq \exp \left\{\varphi\left(\lambda \tau_{\varphi}(\xi)\right)\right\}, \quad \lambda \in \mathbb{R}
		\end{equation}
		holds, where $$
		\tau_{\varphi}(\xi):=\sup _{\lambda \neq 0} \frac{\varphi^{(-1)}(\log \textsf{E} \exp \{\lambda \xi\})}{|\lambda|},
		$$
		and  $\varphi^{(-1)}(\cdot)$ denotes the inverse function of $\varphi(\cdot)$.

	\end{lemma}
	
	\begin{lemma}{\textbf{(\cite{BK} Lemma 4.3)}}
		$Sub_{\varphi}(\Omega)$  is a Banach space with respect to the norm 
		\begin{equation}
			\tau_{\varphi}(X)=\inf \{a \geq 0: \textsf{E} \exp \{\lambda X\} \leq \exp \{\varphi(a \lambda)\}, \,\,\text{for all}\,\,\lambda \in \mathbb{R}\}.
		\end{equation}
		With the norm above, a  $\varphi$-sub-Gaussian process $(X_{t})_{t\in T}$ satisfies the following incremental inequality:
		\begin{equation}\label{2.1}
			\textsf{P}\left(|X_t-X_s| \geqslant u \tau_\varphi(X_t- X_s)\right) \leqslant 2 \exp \left\{-\varphi^*(u)\right\},
		\end{equation}
		where $t,s\in T$.
		
	\end{lemma}
	
	\begin{remark}
		Centered Gaussian  process  $(X_{t})_{t\in T}$ is an important example of  $\varphi$-sub-Gaussian process,  where $\varphi(x)=x^2 / 2$ and $\tau_{\varphi}(X_t)=\left(\textsf{E}|X_t|^2\right)^{1 / 2}$.
	\end{remark}
	
	\begin{remark}
		Another example of $\varphi$-sub-Gaussian distribution lies in the Weibull distribution. When the probability density function of $\xi$ is of the form
		$$f(x ; \kappa, q)= \begin{cases}\frac{q}{\kappa}\left(\frac{x}{\kappa}\right)^{q-1} e^{-(x / \kappa)^q}, & x \geq 0 \\ 0, & x<0\end{cases},$$
where $q>0$ is the shape parameter and $\kappa>1$ is the scale parameter of the distribution, $\xi$ obeys  $\varphi$-sub-Gaussian distribution.  
   \end{remark}

\section{Talagrand's chaining functionals }\label{secT}

Talagrand developed a comprehensive framework on the generic chaining method to capture the relevant geometric structure and the boundness of stochastic processes. Let us begin with some definitions related to generic chaining.  $T$ is the index set of $\varphi$-sub-Gaussian process $(X_t)_{t\in T}$. $d$ is a metric on $T$, and defined by $d(s,t)=\tau_\varphi(X_t- X_s)$.

\begin{definition}
	An increasing sequence $\left(\mathcal{A}_n\right)_{n \geq 0}$ of partitions of $T$ is called admissible if $ \operatorname{card} (\mathcal{A}_0)=1$ and $ \operatorname{card} (\mathcal{A}_n) \leq 2^{2^n}$ for all $n \geq 1$, $ \operatorname{card}(T)$ denotes the cardinality of a set $T$.
\end{definition}

For $t\in T$, $A_n(t)$ denotes the unique element of $\mathcal{A}_n$ which  contains $t$.  If the stochastic process $(X_t)_{t\in T}$ satisfies
\begin{equation}
	\label{sub-1}
	\mathsf{P}\left(|X_t-X_s|\ge u\right)\le 2\exp\left(-\frac{u^2}{2d^2(s,t)}\right),
\end{equation}
where $u>0$, Talagrand introduced chaining functional, $\gamma_{2}(T, d)$, to replace Dudely's bound as the upper bound of the sub-Gaussian process, where 
\begin{eqnarray}\label{g2}
	\gamma_{2}(T, d)=\inf\sup _{t \in T} \sum_{n\geq 0}^{\infty} 2^{n/2} \Delta \left(A_n(t)\right),
\end{eqnarray}
with the infimum taken over all admissible sequences. Here, as always, $ \Delta \left(A_n(t)\right)$ denotes the diameter of a set $A_n(t)$. Inspired by  Talagrand \cite{t2} and Dirksen \cite{SD}, we introduce the following functionals to capture the boundness of  $\varphi$-sub-Gaussian process. For any $1\leq p<\infty$, $k_p:=\left\lfloor\frac{\log(p)}{\log(2)}\right\rfloor$, where $\left\lfloor\cdot \right\rfloor$ denotes the integer part. 

\begin{definition}
	For an Orlicz $N$-function $\varphi$ satisfying $Q-$condition, distribution-dependent Talagrand-type $\gamma$-functional
	is defined by
	\begin{eqnarray}\label{gg1}
		\gamma_{\varphi,p}(T, d)=\inf\sup _{t \in T} \sum_{n\geq k_p}^{\infty} \varphi^{*(-1)}(2^{n}) \Delta \left(A_n(t)\right),
	\end{eqnarray}	
	where the infimum is taken over all admissible sequences.
\end{definition}

When $\varphi(x)=\frac{x^2}{2}$, we denote $\gamma_{\varphi,p}(T, d)$ by $\gamma_{2,p}(T, d)$. We also introduce a modified version of the functional in the following.

\begin{definition}
	For an Orlicz $N$-function $\varphi$ satisfying $Q-$condition,
	distribution-dependent Talagrand-type modified $\gamma$-functional
	is defined by
	\begin{eqnarray} \label{modg}
		\tilde\gamma_{\varphi,p}(T, d)=\inf_{T_n\subset T, \text{card}(T_n)\le 2^{2^n}}\sup _{t \in T} \sum_{n\geq k_p}^{\infty} \varphi^{*(-1)}(2^{n}) d\left(t, T_n\right).
	\end{eqnarray}
\end{definition}

Similarly, when $\varphi(x)=\frac{x^2}{2}$, we denote $\tilde{\gamma}_{\varphi,p}(T, d)$ by $\tilde{\gamma}_{2,p}(T, d)$.

To discuss the properties of (\ref{gg1}) and (\ref{modg}), we need to employ the entropy number of a subset of $T$.  For $A\subset T$,  the entropy number is defined as 
$$e_n(A)=\inf_{S\subset T, \text{card}(S)\le 2^{2^n}}\sup_{t\in A}d(t,S)=\inf\{\varepsilon:N(A, d, \varepsilon)\le 2^{2^n}\},$$
$N(A, d, \varepsilon)$ is the smallest integer $N$ such that $A$ can be covered by $N$ balls of radius $\varepsilon$.

Dudley's bound plays an essential role in the history of studying the boundness of stochastic processes.  It is natural to consider Dudley's bound in our setting. The following theorem shows that our chaining functional is smaller than Dudley's bound. Before presenting our result, we need a condition for $\varphi$.

\begin{definition}
			For any $b\geq1$, if  there exists some $M_b\in(0,1)$, such that
			\begin{equation}\label{0905}
				\frac{\varphi^{*(-1)}(bx)}{\varphi^{*(-1)}(x)}\geq\frac{1}{1-M_b}
			\end{equation}
			holds for any $x\geq 2$, we say that $\varphi$ satisfies  $\Delta_2-$condition.
\end{definition}
		
		\begin{remark}
			One important example satisfying the $\Delta_2-$condition is
			$$\varphi(x)=c|x|^a,\quad a>1,\, c>0.$$
		\end{remark}

\begin{remark}
The definition of  $\Delta_2-$condition is from (8.63) in Talagrand \cite{t2}. When the functions, $-\log \textsf{P}(\xi_i\ge x)$, are convex, Latała \cite{la} studied the the suprema of canonical processes $\sum_{i \ge 1}t_i\xi_i$ under   $\Delta_2-$condition, where $t_i\in \mathbb {R}$. 	$\Delta_2-$condition is essential for obtaining the majorizing measure theorem of  $\sum_{i \ge 1}t_i\xi_i$.	\end{remark}

\begin{theorem}
	It is assumed that the function $\varphi(\cdot)$ satisfies $\Delta_2-$condition. Then, 
	\begin{equation}
		\tilde{\gamma}_{\varphi,p}(T,d)\leq\int_0^{e_{k_p}(T)}\varphi^{*(-1)}\left(N(T,d,\varepsilon)\right)d\varepsilon\leq\int_0^{\infty}\varphi^{*(-1)}\left(N(T,d,\varepsilon)\right)d\varepsilon.
	\end{equation}
	\end{theorem}

\begin{proof}
	It can be observed from the definition of $e_n(T)$ that $N(T,d,\varepsilon)\geq1+2^{2^n}$ as $\varepsilon\leq e_n(T)$. Hence, 
	\begin{equation*}
		\varphi^{*(-1)}(1+2^{2^n})\left(e_n(T)-e_{n+1}(T)\right)\leq\int_{e_n(T)}^{e_{n+1}(T)}	\varphi^{*(-1)}(\log N(T,d,\varepsilon))d\varepsilon.
	\end{equation*}
	Since $\log(1+2^{2^n})\geq2^n\log2$, summation over $n\geq k_p=\lfloor\frac{\log p}{\log2}\rfloor$ gives 
	\begin{equation}\label{dud01}
		\log2\sum_{n\geq k_p}\varphi^{*(-1)}(2^n)(e_n(T)-e_{n+1}(T))\leq\int_0^{e_{k_p}(T)}\varphi^{*(-1)}\left(\log N(T,d,\varepsilon)\right)d\varepsilon.
	\end{equation}
	For some constant $C_1\in(0,1)$, we have,
	\begin{align}\label{dud02}
		\sum_{n\geq k_p}\varphi^{*(-1)}(2^n)(e_n(T)-e_{n+1}(T))=&\sum_{n\geq k_p}\varphi^{*(-1)}(2^n)e_n(T)-\sum_{n\geq k_p+1}\varphi^{*(-1)}(2^{n-1})e_{n}(T\nonumber)\\
		&\geq\sum_{n\geq k_p}\varphi^{*(-1)}(2^n)\left(1-\frac{\varphi^{*(-1)}(2^{n-1})}{\varphi^{*(-1)}(2^{n})}\right)e_n(T)\\
		&\geq M_2\sum_{n\geq k_p}\varphi^{*(-1)}(2^n)e_n(T).\nonumber
	\end{align}
	Thus, a combination of (\ref{dud01}), (\ref{dud02}) and the definition of $\gamma_{\varphi,p}$ yields
	\begin{align}
		\gamma_{\varphi,p}(T,d)\leq\sum_{n\geq k_p}\varphi^{*(-1)}(2^n)e_n(T)&\leq C_2\int_0^{e_{k_p}(T)}\varphi^{*(-1)}\left(\log N(T,d,\varepsilon)\right)d\varepsilon\nonumber\\
		&\leq\int_0^{\infty}\varphi^{*(-1)}\left(\log N(T,d,\varepsilon)\right)d\varepsilon.
	\end{align}
	This completes the proof.
\end{proof}

The growth condition of Talagrand's chaining functional and partitioning schemes is the backbone of  Talagrand's formulation. Now, we discuss the growth condition and partition schemes in our setting.   We first present some definitions, which are from Talagrand \cite{t2}.

\begin{definition}
	A map $F$ is a functional on  subsets of  $T$ if 
	\begin{itemize}
		\item $F(H)\geq0$ for each subset $H$ of $T$;
		\item $F(H')\geq F(H)$ for $H\subset H'\subset T$.
	\end{itemize}
\end{definition}

\begin{definition}
	For a metric space $T$ endowed with distance $d$, given $a>0$ and an integer $r \geq 8$, we say that subsets $H_1, \ldots, H_m$ of $T$ are $(a, r)$-separated if
	$$
	H_{\ell} \subset B\left(t_{\ell}, 2 a / r\right) \quad\text{for all}\quad\ell \leq m,
	$$
	with $ B\left(t_{\ell}, 2 a / r\right)$ denoting a ball centered at $t_{\ell}$ with radius $2a/r$, where the points $t_1, t_2, \ldots, t_m$ in $T$ satisfy
	$$
	a \leq d\left(t_{\ell}, t_{\ell^{\prime}}\right) \leq 2 a r \quad \text{for all}\quad\ell, \ell^{\prime} \leq m, \ell \neq \ell^{\prime} .
	$$
\end{definition}

\begin{definition}
	We say that the functional $F$ satisfies the generalized growth condition with parameters $r \geq 8$ and $c^*>0$ if for any integer $n \geq 1+\lfloor\log_2p\rfloor$ and any $a>0$ the following holds true, where $m=2^{2^n}$ : For each collection of subsets $H_1, \ldots, H_m$ of $T$ that are $(a, r)$-separated, we have
	$$
	F\left(\bigcup_{\ell \leq m} H_{\ell}\right) \geq c^* a \varphi^{*(-1)}\left(2^{n}\right)+\min _{\ell \leq  m} F\left(H_{\ell}\right).
	$$
\end{definition}

\begin{proposition}\label{growth}
	Assume that $r\geq16$,  the functional $\gamma_{\varphi,p}(T, d)$ satisfies the generalized version of the growth condition with parameters $r$ and $c^*=1/8$. Under the same assumption, the functional $\tilde\gamma_{\varphi,p}(T, d)$ also satisfies the generalized version of the growth condition with parameters $r$ and $c^*=1/8$.
\end{proposition}

 
 The proofs of Proposition \ref{growth} and Proposition \ref{growthh} are similar. We only give a detailed proof of Proposition \ref{growth} for simplicity.
 
\begin{proof}
	Consider points $\left(t_{\ell}\right)_{\ell \leq m}$ of $T$ with $d\left(t_{\ell}, t_{\ell^{\prime}}\right) \geq a$ if $\ell \neq \ell^{\prime}$. Consider sets $H_{\ell} \subset B\left(t_{\ell}, a / 8\right)$ and the set $H=\bigcup_{\ell \leq m} H_{\ell}$. For an arbitrary admissible sequence of partitions $\left(\mathcal{A}_n\right)$ of $H$, consider the set
	$$
	I_n=\left\{\ell \leq m, \text{there exists}\,\, A \in \mathcal{A}_{n-1} \,\,\text{satisfying}\,\, A \subset H_{\ell}\right\}.
	$$
	
	An injection $\mathcal{M}: I_n \to \mathcal{A}_{n-1}$ is defined by, for $\ell \in I_n$,
	$$\mathcal{M}(\ell)=A$$
	for an arbitrary $A \in \mathcal{A}_{n-1}$ with $A \subset H_{\ell}$. As a consequence, $ \operatorname{card} I_n \leq \operatorname{card} \mathcal{A}_{n-1} \leq 2^{2^{n-1}}<m=2^{2^n}$. 
	
	Hence, there exists $\ell_0 \notin I_n$, and then $A_{n-1}(t) \not \subset H_{\ell_0}$. So that since $A_{n-1}(t) \subset H$, the set $A_{n-1}(t)$ must intersect a set $H_{\ell} \neq H_{\ell_0}$. Consequently, $A_{n-1}(t)$ intersects the ball $B\left(t_{\ell}, a / 8\right)$. Since $t \in H_{\ell_0}$, we have $d\left(t, B\left(t_{\ell}, a / 8\right)\right) \geq a-2a/r-a/8\geq a / 2$. Since $t \in A_{n-1}(t)$, this implies that $\Delta\left(A_{n-1}(t)\right) \geq a / 2$. 
	
	Since $\Delta\left(A_{n-1}(t) \cap H_{\ell_0}\right) \leq \Delta\left(H_{\ell_0}\right) \leq a / 4$, we have proven that for $t \in H_{\ell_0}$ and any integer $n\geq1$, 
	\begin{equation}\label{equaa1}
		\Delta\left(A_{n-1}(t)\right) \geq \Delta\left(A_{n-1}(t) \cap H_{\ell_0}\right)+\frac{1}{4} a.
	\end{equation}
	
	Recall that for $p\geq1$, $ \lfloor\frac{\log p}{\log2}\rfloor$ is denoted by $k_p$. Now, since for each $k \geq 0$ we have $\Delta\left(A_k(t)\right) \geq \Delta\left(A_k(t) \cap H_{\ell_0}\right)$, we have
	$$
	\begin{aligned}
		&\sum_{k \geq k_p} \varphi^{*(-1)}\left(2^{k }\right)\left(\Delta\left(A_k(t)\right) -\Delta\left(A_k(t) \cap H_{\ell_0}\right)\right)\\
		& \geq \varphi^{*(-1)}\left(2^{n+k_p-1}\right)\left(\Delta\left(A_{n+k_p-1}(t)\right)-\Delta\left(A_{n+k_p-1}(t) \cap H_{\ell_0}\right)\right) \\
		& \geq \frac{1}{4} a \varphi^{*(-1)}\left(2^{n-1}\right),
	\end{aligned}
	$$
	where we have used (\ref{equaa1}) in the last inequality, and, consequently,
	\begin{equation}\label{equaa2}
		\sum_{k\geq k_p} \varphi^{*(-1)}\left(2^{k}\right) \Delta\left(A_k(t)\right) \geq \frac{1}{4} a \varphi^{*(-1)}\left(2^{n-1}\right)+\sum_{k\geq k_p} \varphi^{*(-1)}\left(2^{k }\right) \Delta\left(A_k(t) \cap H_{\ell_0}\right).
	\end{equation}
	
	Next, consider the admissible sequence $\left(\mathcal{A}_n^{\prime}\right)$ of $H_{\ell_0}$ given by $\mathcal{A}_n^{\prime}=\{A \cap$ $\left.H_{\ell_0} ; A \in \mathcal{A}_n\right\}$. We have, by definition,
	$$
	\sup _{t \in H_{\ell_0}} \sum_{k \geq k_p} \varphi^{*(-1)}\left(2^{k}\right) \Delta\left(A_k(t) \cap H_{\ell_0}\right) \geq \gamma_{\varphi,p}\left(H_{\ell_0}, d\right) .
	$$
	
	Hence, taking the supremum over $t$ in $H_{\ell_0}$ in (\ref{equaa2}), we get
	\begin{align*}
		\sup _{t \in H_{\ell_0}} \sum_{k \geq k_p} \varphi^{*(-1)}\left(2^{k }\right) \Delta\left(A_k(t)\right) &\geq \frac{1}{4} a  \varphi^{*(-1)}\left(2^{n-1}\right)+\gamma_{\varphi,p}\left(H_{\ell_0}, d\right)\\
		&\geq \frac{1}{8} a  \varphi^{*(-1)}\left(2^{n }\right)+\min _{\ell \leq m} \gamma_{\varphi,p}\left(H_{\ell}, d\right).
	\end{align*}
	Since the admissible sequence $\left(\mathcal{A}_n\right)$ is arbitrary, we have proven that
	\begin{equation}
		\gamma_{\varphi,p}(H, d) \geq \frac{1}{8} a \varphi^{*(-1)}\left(2^{n}\right)+\min _{\ell \leq m} \gamma_{\varphi,p}\left(H_{\ell}, d\right),
	\end{equation}
	which completes the proof.
\end{proof}

The following result is on the partitioning schemes. It is essential for obtaining the lower bound or other important results in generic chaining.  The  equivalence of $\gamma_{\varphi,p}(T, d)$ and $\tilde\gamma_{\varphi,p}(T, d)$ can be also derived by this result. See Theorem \ref{eq1} for the proof of equivalence.

\begin{theorem}\label{th 2.9.1}
	Assume a functional $F$ exists, which satisfies the generalized growth condition on $T$ for the parameters $r$ and $c^*$. Assume that the function $\varphi(\cdot)$ satisfies $\Delta_2-$condition, 
	then we have 
	\begin{equation}\label{gc}
		\gamma_{\varphi}\left(T , d\right) \leqslant \frac{L r}{c^*} F(T)+ L r \Delta(T).
	\end{equation}
\end{theorem}

Some lemmas are needed for the proof of Theorem \ref{th 2.9.1}. The following two elementary lemmas are from Talagrand \cite{t2}.

\begin{lemma}\label{l2.9.3}
	For a given metric space $(T,d)$, assume that any sequence $\left(t_{\ell}\right)_{\ell \leq m}$ with $d\left(t_{\ell}, t_{\ell^{\prime}}\right) \geq$ a for $\ell \neq \ell^{\prime}$ satisfies $m \leq$ $N$. Then, $T$ can be covered by $N$ balls of radius $a$.	
\end{lemma}

\begin{lemma}\label{l2.9.5}
	Consider numbers $\left(a_n\right)_{n \geq 0}, a_n>0$, and assume $\sup _n a_n<\infty$. Consider $\alpha>1$, and define
	$$
	I=\left\{n \geq 0 ; \text{for all}\,\, k \geq 0, k \neq n, a_k<a_n \alpha^{|n-k|}\right\} .
	$$
	Then $I \neq \emptyset$, and we have
	$$
	\sum_{k \geq 0} a_k \leq \frac{2 \alpha}{\alpha-1} \sum_{n \in I} a_n.
	$$
\end{lemma}

The following lemma provides a fundamental principle for the proof of partition schemes. 

\begin{lemma}\label{l2.9.4}
	Under the conditions of Theorem \ref{th 2.9.1}, consider $B \subset T$ satisfies that $\Delta(B) \leq 2 r^{-j}$  for a certain $j \in \mathbb{Z}$. For any $n \geq 0$, set $m=2^{2^n}$. Then there exists a partition $\left(A_{k}\right)_{k \leq m}$ of $B$ into sets which have either of the following properties:
	\begin{equation}\label{2.9.4.01}
		\Delta\left(A_{k}\right) \leq 2 r^{-j-1},
	\end{equation}
	or else
	\begin{equation}\label{2.9.4.02}
		t \in A_{k} \Rightarrow F\left(B \cap B\left(t, 2 r^{-j-2}\right)\right) \leq F(B)-c^* \varphi^{*(-1)}(2^{n}) r^{-j-1} .
	\end{equation}
\end{lemma}

\begin{proof}
	Consider the set below
	$$
	I=\left\{t \in B ; F\left(B \cap B\left(t, 2 r^{-j-2}\right)\right)>F(B)-c^* \varphi^{*(-1)}(2^{n}) r^{-j-1}\right\} .
	$$
	For points $\left(t_{k}\right)_{k \leq m^{\prime}}$ in $I$ such that $d\left(t_{k}, t_{k^{\prime}}\right) \geq r^{-j-1}$ when $k \neq k^{\prime}$, we have $m^{\prime}<m$. For otherwise, by the generalized growth condition, for $a=r^{-j-1}$ and for the sets $H_{k}:=$ $B \cap B\left(t_{k}, 2 r^{-j-2}\right)$, 
	$$
	F(B) \geq F\left(\bigcup_{k \leq m} H_{k}\right) \geq c^* \varphi^{*(-1)}(2^{n}) r^{-j-1}+\min _{k \leq m} F\left(H_{k}\right)>F(B) .
	$$
	The contradiction means $m^{\prime}<m$.  Consequently, then by the Lemma \ref{l2.9.3} with $N=m-1$, we may cover $I$ by $m^{\prime}<m$ balls $\left(B_{k}\right)_{k \leq m^{\prime}}$ of radius $\leq r^{-j-1}$. Finally, we set $A_{k}=I \cap\left(B_{k} \backslash \cup_{k^{\prime}<k} B_{k^{\prime}}\right)$ for $k \leq m^{\prime}, A_{k}=\emptyset$ for $m^{\prime}<k<m$ and $A_m=B \backslash I$, and the proof of the lemma is complete.
\end{proof}

\begin{proof}[Proof of Theorem \ref{th 2.9.1}]

	To drive (\ref{gc}), we construct an admissible sequence of partitions $\mathscr{A}_n$, and for $A\in \mathscr{A}_n$, we construct $j_n(A) \in Z$, such that 
	\begin{equation}
		\Delta(A) \leq 2 r^{-j_n(A)}.
	\end{equation}
	
	We start with $\mathcal{A}_0=\{T\}$ and $j_0(T)=\max\{j_0\in\mathbb{Z}, \Delta(T)\leq2 r^{-j_0}\}$. Having constructed $\mathcal{A}_n$, we construct $\mathcal{A}_{n+1}$ as follows: for each $B \in \mathcal{A}_n$, we use Lemma \ref{l2.9.4} with $j=j_n(B)$ to split $B$ into sets $\left(A_{\ell}\right)_{\ell \leq 2^{2^n}}$. If $A_{\ell}$ satisfies (\ref{2.9.4.01}), we set $j_{n+1}\left(A_{\ell}\right)=j_n(B)+1$, and otherwise we set $j_{n+1}\left(A_{\ell}\right)=j_n(B)$. 
	
	The sequence thus constructed is admissible, since each set $B$ in $\mathcal{A}_n$ is split into at most $2^{2^n}$ sets and $({2^{2^n}})^2 \leq 2^{2^{n+1}}$. We note also by construction that if $B \in \mathcal{A}_n$ and $A \subset B, A \in \mathcal{A}_{n+1}$, then
	\begin{itemize}
		\item Either $j_{n+1}(A)=j_n(B)+1$,
		\item Or else $j_{n+1}(A)=j_n(B)$ and
		\begin{equation}\label{largepiece}
			F(B \cap B(t,2 r^{-j_{n+1}(A)-2})) \leq F(B)-c^* \varphi^{*(-1)}(2^n) r^{- j_{n+1}(A)-1}.
		\end{equation}
	\end{itemize}
	
	Now, it suffices to show that for a fixed $t \in T$,
	\begin{equation}\label{firststep}
		\sum_{n \geqslant k_p} \varphi^{*(-1)}(2^{n}) \Delta\left(A_n(t)\right) \leqslant \frac{L r}{c^*} F(T)+Lr \Delta(T).
	\end{equation}
	
	Let $j(n)=j_n(A_n(t))$, and by the construction of the admissible sequence $\mathcal{A}_n$, we know that 
	\begin{equation}\label{DeltaA}
		\Delta(A_n(t)) \leq 2 r^{-j(n)}.
	\end{equation}
	
	Let $a(n)=\varphi^{*(-1)}\left(2^n\right) r^{-j(n)}$,  and $M_2$ be the constant defined in (\ref{0905}) and then take 
	$$
	I=\left\{n \geqslant 0; \text{for all}\,\, k \geqslant 0 ,  n \neq k , a(k)<a(n)\left(\frac{1}{1-M_2}\right)^{|n-k|}\right\}.
	$$
	
	Then by Lemma \ref{l2.9.5},
	$$
	\sum_{n \geqslant k_p} \varphi^{*(-1)}(2^{n}) \Delta\left(A_n(t)\right)\leq\sum_{n \geqslant 0} \varphi^{*(-1)}(2^{n}) \Delta\left(A_n(t)\right) \leq \frac{2}{M_2}\sum_{n \in I} \varphi^{*(-1)}(2^{n}) \Delta\left(A_n (t)\right) .
	$$
	
	Observe that by the definition of $j_0$, we know that $2r^{-j_0}\leq r\Delta(T)$. Hence, $a(0)=\varphi^{*(-1)}(1)r^{-j_0}\leq \varphi^{*(-1)}(1)r\Delta(T)/2$, so we only need to show that  
	\begin{equation}\label{secondstep}
		\sum_{n \in I \backslash\{0\}}\varphi^{*(-1)}(2^{n}) \Delta\left(A_n (t)\right)\leq\frac{1}{2}\sum_{n \in I \backslash\{0\}} a(n) \leqslant \frac{L r}{c^*} F(T).
	\end{equation}
	
	Note that if $n \in I $, then 
	$$a(n+1)<\frac{1}{1-M_2} a(n),\quad a(n-1)<\frac{1}{1-M_2} a(n).$$
	Also, we know that 
	$$
	a(n+1)=\frac{\varphi^{*(-1)}\left(2^{n+1}\right)}{\varphi^{*(-1)}\left(2^n\right)} r^{j(n)-j(n+1)}a(n) .
	$$
	
	Therefore, 
	\begin{align*}
		\frac{1}{1-M_2} r^{j(n)-j(n+1)}a(n)&<\frac{\varphi^{*(-1)}\left(2^{n+1}\right)}{\varphi^{*(-1)}\left(2^n\right)} r^{j(n)-j(n+1)}a(n)\\
		&=a(n+1)<\frac{1}{1-M_2} a(n),
	\end{align*}
	
	\begin{align*}
		a(n-1)<\frac{1}{1-M_2}a(n)&=\frac{1}{1-M_2}\frac{\varphi^{*(-1)}\left(2^{n}\right)}{\varphi^{*(-1)}\left(2^{n-1}\right)} r^{j(n-1)-j(n)}a(n-1)\\
		&\leq\left(\frac{\varphi^{*(-1)}\left(2^{n}\right)}{\varphi^{*(-1)}\left(2^{n-1}\right)}\right)^2 r^{j(n-1)-j(n)}a(n-1)\\
		&\leq4r^{j(n-1)-j(n)}a(n-1).
	\end{align*}
	
	The observation above implies that $j(n+1)=j(n)+1 $ and $ j(n-1)=j(n)$.
	Hence, when $n \geq k_p$, then let $I \backslash\{0\}$ enumerated as as $n_1<n_2<\ldots$, so that  
	\begin{equation}\label{2.9.2}
		j\left(n_k+1\right)=j\left(n_k\right)+1 ;\quad j\left(n_k-1\right)=j\left(n_k\right) .
	\end{equation}
	As a consequence, 
	\begin{equation}\label{jjj}
		j\left(n_{k+2}\right) \geq j\left(n_{k+1}\right)+1 \geq j\left(n_k\right)+2.
	\end{equation}
	
	For simplicity, denote $F\left(A_n(t)\right)$ by  $f(n)$. Then $f(0)=F(T)$ and the sequence $(f(n))_{n\geq0}$ is decreasing because $A_n(t)\subset A_{n-1}(t)$. Then, the key to derive (\ref{secondstep}) is to prove that for $k\geq1$, 
	\begin{equation}\label{thirdstep}
		a\left(n_k\right) \leq \frac{L r}{c^*}\left(f\left(n_{k}-1\right)-f\left(n_{k+2}\right)\right).
	\end{equation}
	For $k\geq2$, $f(n_k-1)\leq f(n_{k-1})$, so that (\ref{thirdstep}) implies that 
	\begin{equation*}
		a\left(n_k\right) \leq \frac{L r}{c^*}\left(f\left(n_{k-1}\right)-f\left(n_{k+2}\right)\right).
	\end{equation*}
	Summation overall $k\geq2$ gives 
	\begin{equation}\label{sumkgeq2}
		\sum_{k\geq2}a(n_k)\leq\frac{Lr}{c^*}F(T),
	\end{equation}
	and then a combination of (\ref{sumkgeq2}) and (\ref{thirdstep}) for $k=1$ concludes the proof of the theorem when $I$ is infinite.
	The case when $I$ is finite will be proven at the end of the proof.
	
	So now let us roll up our sleeves to handle (\ref{thirdstep}).
	
	Since $n_k \geq 1$, applying (\ref{largepiece}) to $A=A_{n_k}(t)$ and $B=A_{n_k-1}(t)$, thus we get
	\begin{equation}
		\begin{aligned}
			& F\left(B \cap B\left(t, 2 r^{-j_{n_k}(A)-2}\right)\right) \\
			& \leqslant F(B)-c^* \varphi^{*(-1)}\left(2^{n_k-1}\right) r^{-j_{n_k}(A)-1} \\
			&=F(B)-c^* \frac{\varphi^{*(-1)}\left(2^{n_k-1}\right)}{\varphi^{*(-1)}\left(2^{n_k}\right)} \varphi^{* (-1)}\left(2^{n_k}\right) r^{-j_{n_k}(A)-1} \\
			& =F(B)-c^* \frac{\varphi^{*(-1)}\left(2^{n_k-1}\right)}{\varphi^{*(-1)}\left(2^{n_k}\right)}  a(n_k) \cdot r^{-1}.
		\end{aligned}
	\end{equation}
	
	Combined with Lemma \ref{inv},  we obtain that
	\begin{align}\label{2.9.7}
		a\left(n_k\right) &\leq \frac{ r}{c^*}\frac{\varphi^{*(-1)}\left(2^{n_k}\right)}{\varphi^{*(-1)}\left(2^{n_k-1}\right)}  \left(F(B)-F\left(B \cap B\left(t, 2 r^{-j_{n_k}(A)-2}\right)\right)\right)\nonumber\\
		&\leq \frac{2r}{c^*} \left(F(B)-F\left(B \cap B\left(t, 2 r^{-j_{n_k}(A)-2}\right)\right)\right).
	\end{align}

	Moreover, recalling that $j\left(n_{k+2}\right)=j_{n_{k+2}}\left(A_{n_{k+2}}(t)\right)$, hence a combination with (\ref{jjj}) gives 
	$\Delta\left(A_{n_{k+2}}(t)\right) \leq 2 r^{-j\left(n_{k+2}\right)} \leq 2 r^{-j\left(n_k\right)-2}$. So $A_{n_{k+2}}(t) \subset B \cap B\left(t, 2 r^{-j\left(n_k\right)-2}\right)$ and then
	\begin{align}\label{fnk+2}
		f\left(n_{k+2}\right)\leq F\left(B \cap B\left(t, 2 r^{-j_{n_k}(A)-2}\right)\right).
	\end{align}
	
	Recalling that $F(B)=f(n_k-1)$, combined with (\ref{2.9.7}) and (\ref{fnk+2}), we have 
	$$
	a\left(n_k\right) \leqslant \frac{L r}{c^*}\left(f\left(n_{k}-1\right)-f\left(n_{k+2}\right)\right).
	$$
	then we finish the proof of (\ref{gc}) when $I$ is infinite.
	
	When $I$ is finite, we denote the largest element of $I$ as $n_{\bar{k}}$. For $k\leq\bar{k}-2$, $a(n_k)$ is controlled by previous argument. For $k=\bar{k}-1$ and $k=\bar{k}$, we use the fact that for $n\geq0$,
	\begin{equation}\label{onean}
		a(n)\leq\frac{Lr}{c^*}F(T)+L\Delta(T).
	\end{equation}
	
	For $n\geq1$ and $j(n-1)=j(n)$, use (\ref{largepiece}) for $n-1$ rather than $n$ yields (\ref{onean}). For $n\geq1$ and $j(n-1)=j(n)-1$,
	\begin{align*}
		a(n)&=\frac{\varphi^{*(-1)}\left(2^{n}\right)}{\varphi^{*(-1)}\left(2^{n-1}\right)}r^{-1}a(n-1)\\
		&\leq\frac{2}{r}a(n-1)<a(n-1).
	\end{align*}
	Iterate this relation until we reach an integer $n^*$ with 
	\begin{itemize}
		\item either $j(n^*)=j(n^*-1)$,
		\item or $n^*=0$.
	\end{itemize}
	
	Note that $a(0)\leq L\Delta(T)$, the case when $I$ is finite is done. We finally prove the theorem.
\end{proof}

With Theorem $\ref{th 2.9.1}$ in hand, we obtain the equivalence of $\gamma_{\varphi,p}(T, d)$ and $\tilde\gamma_{\varphi,p}(T, d)$ under proper conditions.

\begin{theorem}\label{eq1}
	For an Orlicz $N$-function $\varphi$ satisfying $Q-$condition and $\Delta_2-$condition,
	$$\gamma_{\varphi,p}(T, d)\asymp \tilde\gamma_{\varphi,p}(T, d).$$
\end{theorem}

\begin{proof}
		It is easy to see that 
	$$ \tilde\gamma_{\varphi,p}(T, d)\lesssim \gamma_{\varphi,p}(T, d). $$
	Hence, it suffices to prove that 
    $$\gamma_{\varphi,p}(T, d)\lesssim \tilde\gamma_{\varphi,p}(T, d). $$
    By Proposition \ref{growthh}, we know that $\tilde\gamma_{\varphi,p}(T,d)$ satisfies the generalized growth condition. By Theorem \ref{th 2.9.1} , we know that  
\begin{align*}
 \gamma_{\varphi}\left(T , d\right) &\leq \frac{L r}{c^*} \tilde\gamma_{\varphi,p}(T,d)+ L r \Delta(T)\\
 &\lesssim \tilde\gamma_{\varphi,p}(T, d). 
\end{align*}
This finally completes the proof.

\end{proof}

\section{ Upper bounds and tails bounds via generic chaining}\label{secR}
\setcounter{equation}{0}
\renewcommand{\theequation}{\thesection.\arabic{equation}}

In the following,  $||\cdot||_p\,\,(p\geq1)$ denotes the $L_p$ norm of a random variable. $T$ is the index set of $\varphi$-sub-Gaussian process $(X_t)_{t\in T}$. $d$ is a metric on $T$, and defined by $d(s,t)=\tau_\varphi(X_t- X_s)$. $\displaystyle\Delta(S):=\sup_{s,t\in S}d(s,t)$ denotes the diameter of the a set $S\in T$ with metric $d$. For an admissible sequence $T_n$, $\pi_n(t):=\mathop{\arg\min}\limits_{s \in T_n} d(s, t)$. $a\vee b=\max\{a,b\}$ for $a,b\in\mathbb{R}$. $\langle\cdot,\cdot\rangle$ represents inner product. $A^t$ denotes the transpose of a matrix $A$.

\begin{theorem}\label{dks}
	For $\varphi$-sub-Gaussian process $(X_{t})_{t\in T}$, we have, for any given $1 \leq p<\infty$, 
	\begin{equation}\label{p3.1}
		\begin{aligned}
			\left(\textsf{E}\sup_{t \in T}|X_t|^p\right)^{\frac{1}{p}}\leq C_0\tilde\gamma_{\varphi,p}(T,d)+\inf_{t_0\in T}\left\{2\sup_{t\in T}\left(\textsf{E}|X_{t}-X_{t_0}|^p\right)^{\frac{1}{p}}+(\textsf{E}|X_{t_0}|^p)^{1/p}\right\}.
		\end{aligned}
	\end{equation}
	Here, $C_0$ is a universal constant. As a consequence, we have,
	\begin{equation}\label{p3.1}
		\begin{aligned}
			\left(\textsf{E}\sup_{t \in T}|X_t|^p\right)^{\frac{1}{p}}\leq C_1\tilde\gamma_{\varphi,p}(T,d)+C_2\sqrt{p}\Delta(T)+C_3p\Delta(T),
		\end{aligned}
	\end{equation}
	with $C_1, C_2, C_3$ some positive universal constants. Therefore, for $u\geq\sqrt{2},$
	\begin{equation}\label{t06}
		\begin{aligned}
			\textsf{P}\left(\sup_{t\in T}\left|X_t-X_{t_0}\right|\geq C_4\Delta(T)\varphi^*(u)+C_5\Delta(T)\sqrt{\varphi^*(u)}+C_6\tilde\gamma_{\varphi,p}(T,d)\right)\leq\exp(-\varphi^*(u)),
		\end{aligned}
	\end{equation}
	where $C_4, C_5, C_6$ are positive universal constants.
\end{theorem}

\begin{corollary}\label{signalrlemma}
	For  $\varphi$-sub Gaussian process $(X_t)_{t\in T}$, we have, for any given $1\leq p<\infty$,
	\begin{equation*}
		(\textsf{E}|X_t|^p)^{1/p}\leq C_4\Delta(T)p+C_4\Delta(T)p^{1/2}.
	\end{equation*}
\end{corollary}

Before starting to prove our results, we need some lemmas.

\begin{lemma}\label{DA1} (Foucart and Rauhut  \cite{FR})
	If a random variable $\xi$ satisfies
	$$\left(\textsf{E}\left|\xi\right|^p\right)^{\frac{1}{p}} \leq c_1 p+c_2 \sqrt{p}+c_3, \quad \text { for all } \quad p \geq 1,$$
	for some $0 \leq c_1, c_2, c_3<\infty$, then
	$$\textsf{P}\left(\left|\xi\right| \geq e\left(c_1 u+c_2 \sqrt{u}+c_3\right)\right) \leq \exp (-u) \quad(u \geq 1).$$
\end{lemma}

\begin{lemma}\label{EJPA.5} (Dirksen \cite{SD})
	Fix $1 \leq p<\infty$ and $0<\alpha<\infty$. Let $\gamma \geq 0$ and suppose that $\xi$ is a positive random variable such that for some $c \geq 1$ and $u_*>0$,
	$$
	\textsf{P}(\xi>\gamma u) \leq c \exp \left(-p u^\alpha / 4\right) \quad\left(u \geq u_*\right) .
	$$
	Then, for a constant $\tilde{c}_\alpha>0$ depending only on $\alpha$,
	$$
	\left(\textsf{E} \xi^p\right)^{1 / p} \leq \gamma\left(\tilde{c}_\alpha c+u_*\right) .
	$$
\end{lemma}

\begin{lemma}\label{moment}
	For $g\sim N(0,1)$, for any $\alpha\geq1$,
	$$\textsf{E}\left|g\right|^{\alpha}\leq\sqrt{\frac{e}{e-1}}\alpha^{\frac{\alpha}{2}}\quad.$$
\end{lemma}

\begin{proof}
	It is easily checked that for $x>0$,
	$$\left(\frac{x}{\sqrt{\alpha}}\right)^{\alpha}\leq\exp\left(\frac{x^2}{2e}\right).$$
	Then
	\begin{equation*}
		\begin{aligned}
			\frac{\textsf{E}\left|g\right|^{\alpha}}{\alpha^{\frac{\alpha}{2}}}=\frac{2}{\sqrt{2\pi}}\int_0^{\infty}\left(\frac{x}{\sqrt{\alpha}}\right)^{\alpha}\exp\left(-\frac{x^2}{2}\right)dx=\sqrt{\frac{e}{e-1}}.
		\end{aligned}
	\end{equation*}
	This completes the proof.
\end{proof}

\begin{proof}[Proof of Theorem \ref{dks}]\label{pf1}
	By triangle inequality, we have for any $t_0\in T$,
	\begin{equation}\label{t01}
		\begin{aligned}
			\displaystyle\left(\textsf{E}\sup_{t \in T}|X_t|^p\right)^{\frac{1}{p}} &\leq\left(\textsf{E}\sup_{t \in T}|X_t-X_{t_0}|^p\right)^{\frac{1}{p}}+\left(\textsf{E}|X_{t_0}|^p\right)^{\frac{1}{p}}  \\
			&\leq\left(\textsf{E} \sup_{t \in T}|X_t-X_{\pi_{k_p}(t)}|^p\right)^{\frac{1}{p}}+\left(\textsf{E} \sup_{t \in T}|X_{\pi_{k_p}(t)}-X_{t_0}|^p\right)^{\frac{1}{p}}+\left(\textsf{E}|X_{t_0}|^p\right)^{\frac{1}{p}} .
		\end{aligned}
	\end{equation}
	
	For the second term in (\ref{t01}),
	\begin{equation}\label{t02}
		\begin{aligned}
			\left(\textsf{E} \sup_{t \in T}|X_{\pi_{k_p}(t)}-X_{t_0}|^p\right)^{\frac{1}{p}}&\leq\left(\textsf{E} \sum_{t \in T}|X_{\pi_{k_p}(t)}-X_{t_0}|^p\right)^{\frac{1}{p}}\\
			&\leq\left(2^{2^{k_p}}\sup_{t\in T}\textsf{E}|X_{\pi_k(t)}-X_{t_0}|^p\right)^{\frac{1}{p}}\\
			&\leq2\sup_{t\in T}\left(\textsf{E}|X_{\pi_{k_p}(t)}-X_{t_0}|^p\right)^{\frac{1}{p}}\leq2\sup_{t\in T}\left(\textsf{E}|X_t-X_{t_0}|^p\right)^{\frac{1}{p}}.
		\end{aligned}
	\end{equation}
	
	The generic chaining method is needed to bound the first term in (\ref{t01}). Recall that the increment condition here is 
	$$\textsf{P}\left(\left|X_t-X_s\right|\geq ud(t,s)\right)\leq2\exp(-\varphi^*(u)).$$
	
	Hence, we have 
	\begin{equation}\label{lstart}
		\begin{aligned}
			\textsf{P}\left(\left|X_{\pi_n(t)}-X_{\pi_{n-1}(t)}\right|\geq\varphi^{*(-1)}(2^nu)d(\pi_n(t),d_{n-1}(t))\right)\\
			\leq2\exp\left(-\varphi^*(\varphi^{*(-1)}(2^nu))\right)=2\exp(-2^nu).
		\end{aligned}
	\end{equation}
	
	Here, $ \operatorname{card}\{(\pi_n(t),\pi_{n-1}(t));t\in T\}\leq  \operatorname{card}(T_n)\,\, \operatorname{card}(T_{n-1})\leq2^{2^n}2^{2^{n-1}}\leq2^{2^{n+1}}$. $\Omega_{u,p,n}$ denotes the following event:
	$$\left|X_{\pi_n(t)}-X_{\pi_{n-1}(t)}\right|\leq\varphi^{*(-1)}(2^nu)d(\pi_n(t),\pi_{n-1}(t)) \quad\text{for all} \quad t\in T.$$
	
	By the choice of $k$, we have $2^{k_p}\leq p<2^{{k_p}+1}$. Then we know that for $u\geq2$,
	\begin{equation*}
		\begin{aligned}
			\textsf{P}\left((\bigcap_{n>k_p}\Omega_{u,p,n})^{c}\right)&\leq2\sum_{n>k_p}2^{2^n+1}\exp\left(-2^nu\right)=2\sum_{n>k_p}\exp\left(2(\log2)2^n\right)\exp(-2^nu)\\
			&\leq 2\sum_{n>k_p}\exp\left((\log2-1)u2^n\right)=2\exp\left(-\frac{2^{k_p}u}{2}\right)\sum_{n>k_p}\exp\left((\log2-1)u2^n+\frac{2^{k_p}u}{2}\right)\\
			&\leq2\exp\left(-\frac{2^{k_p}u}{2}\right)\sum_{n\geq k_p}\exp\left((\log2-1)u2^n+\frac{2^nu}{4}\right)\\
			&\leq c\exp\left(-\frac{pu}{4}\right).
		\end{aligned}
	\end{equation*}
	Here, 
	 $$ c \leq2\sum_{n\geq k_p}\exp\left((\log2-1)u2^n+\frac{2^nu}{4}\right) \leq2\sum_{n\geq 0}\exp\left(2(\log2-3/4)n\right)<\infty. $$

	If event $\displaystyle\bigcap_{n>k_p}\Omega_{u,p,n}$ occurs, then for $u\geq2$, by Lemma \ref{inv},
	\begin{equation*}
		\begin{aligned}
			\left|\sum_{n>k_p}X_{\pi_n(t)}-X_{\pi_{n-1}(t)}\right|&\leq\sum_{n>k}\left|X_{\pi_n(t)}-X_{\pi_{n-1}(t)}\right|\\
			&\leq \sum_{n>k_p}\varphi^{*(-1)}(2^nu)d(\pi_n(t),\pi_{n-1}(t))\\
			&\leq u\sum_{n>k_p}\varphi^{*(-1)}(2^n)d(\pi_n(t),t)+u\sum_{n>k_p}\varphi^{*(-1)}(2^n)d(t,\pi_{n-1}(t))\\
			&\leq u\sum_{n>k_p}\varphi^{*(-1)}(2^n)d(\pi_n(t),t)+2u\sum_{n\geq k_p}\varphi^{*(-1)}(2^n)d(t,\pi_{n}(t))\\
			&\leq3u\tilde{\gamma}_{\varphi,p}(T,d).
		\end{aligned}
	\end{equation*}
	Taking the supremum on both sides, then we get for $u \geq 2$,
	$$\sup_{t \in T}|X_t-X_{\pi_{k_p}(t)}|\leq 3u\tilde{\gamma}_{p}(T, d).$$
	
	To conclude, 
	\begin{equation}\label{lend}
		\begin{aligned}
			\textsf{P}\left(\sup_{t\in T}\left|X_t-X_{\pi_{k_p}(t)}\right|\geq3u\tilde{\gamma}_{\varphi,p}(T,d)\right)\leq c\exp\left(-\frac{pu}{4}\right)\qquad(u\geq2).
		\end{aligned}
	\end{equation}
	
	Therefore, by Lemma \ref{moment},
	\begin{equation*}
		\begin{aligned}
			\textsf{E}\sup_{t\in T}\left|X_t-X_{\pi_{k_p}(t)}\right|^p&=\int_{0}^{\infty}pu^{p-1}\textsf{P}\left(\sup_{t\in T}\left|X_t-X_{\pi_{k_p}(t)}\right|>u\right)du\\
			&=\int_0^{\infty}p(3\tilde{\gamma}_{\varphi,p}(T,d))^p v^{p-1}\textsf{P}\left(\sup_{t\in T}\left|X_t-X_{\pi_{k_p}(t)}\right|>3v\tilde{\gamma}_{\varphi,p}(T,d)\right)dv\\
			&=(3\tilde{\gamma}_{\varphi,p}(T,d))^p\left(\int_0^2+\int_2^{\infty}\right)pv^{p-1}\textsf{P}\left(\sup_{t\in T}\left|X_t-X_{\pi_{k_p}(t)}\right|>3v\tilde{\gamma}_{\varphi,p}(T,d)\right)dv\\
			&\leq(3\tilde{\gamma}_{\varphi,p}(T,d))^p\left(2^p+cp\int_2^{\infty}v^{p-1}\exp\left(-\frac{pv}{4}\right)dv\right)\\
			&\leq(6\tilde{\gamma}_{\varphi,p}(T,d))^p+cp(3\tilde{\gamma}_{\varphi,p}(T,d))^p\int_0^{\infty}v^{p-1}\exp\left(-\frac{pv}{4}\right)dv\\
			&=(6\tilde{\gamma}_{\varphi,p}(T,d))^p+2^{1+p}p^{1-p}c(3\tilde{\gamma}_{\varphi,p}(T,d))^p\int_0^{\infty}u^{2p-1}\exp\left(-\frac{u^2}{2}\right)du\\
			&=(6\tilde{\gamma}_{\varphi,p}(T,d))^p+2^{1+p}p^{1-p}c(3\tilde{\gamma}_{\varphi,p}(T,d))^p\frac{\sqrt{2\pi}}{2}\textsf{E}|g|^{2p-1}\\
			&\leq(6\tilde{\gamma}_{\varphi,p}(T,d))^p+2^{p}p^{1-p}(3\tilde{\gamma}_{\varphi,p}(T,d))^p\frac{c\sqrt{2\pi e}}{\sqrt{e-1}}(2p-1)^{p-\frac{1}{2}}\\
			&\leq(6\tilde{\gamma}_{\varphi,p}(T,d))^p+2^{2p}(3\tilde{\gamma}_{\varphi,p}(T,d))^p\frac{c\sqrt{\pi e}}{\sqrt{e-1}}p^{\frac{1}{2}}.
		\end{aligned}
	\end{equation*}
	Here, $g$ is a standard Gaussian random variable.
	
	Noting that $p^{1/2p}\leq e^{1/2e}$, we obtain that 
	\begin{equation}\label{t03}
		\begin{aligned}
			\left(\textsf{E}\sup_{t\in T}\left|X_t-X_{\pi_{k_p}(t)}\right|^p\right)^{\frac{1}{p}}&\leq\left((3\tilde{\gamma}_{\varphi,p}(T,d))^p(2^p+2^{2p}\frac{c\sqrt{\pi e}}{\sqrt{e-1}}p^{\frac{1}{2}})\right)^{\frac{1}{p}}\\
			&\leq3\tilde{\gamma}_{\varphi,p}(T,d)\left(2+4\left(\frac{c\sqrt{\pi e}}{\sqrt{e-1}}\right)^{\frac{1}{p}}p^{\frac{1}{2p}}\right)\\
			&\leq C_1\tilde{\gamma}_{\varphi,p}(T,d),
		\end{aligned}
	\end{equation}
	for some universal constant $C_1>0$.
	
	By a standardization of the Orlicz $N$-function $\varphi$ in a neighborhood of zero (see \cite{BK} P.67 for details), we might assume $\varphi(x)=\frac{x^2}{2}$ for $|x|<\sqrt{2}$.
	
	By Lemma \ref{book}, $\frac{\varphi^*(u)}{u}$ is monotonically nondecreasing for $u>0$, 
	$$\varphi^*(u)\geq\varphi^*(1)u=\frac{u}{2} \quad\text{for}\quad u>1.$$
	
	By the increment condition, we get
	\begin{equation}\label{appsta}
		\begin{aligned}
			\textsf{P}\left(\left|X_t-X_{t_0}\right|\geq u\Delta(T)\right)&\leq\textsf{P}\left(\left|X_t-X_s\right|\geq ud(t,s)\right)\\
			&\leq2\exp\left(-\varphi^*(u)\right)=\begin{cases}2\exp\left(-\frac{u^2}{2}\right) \quad0\leq u<1\\
				2\exp\left(-\frac{u}{2}\right)\quad u\geq1\\
			\end{cases}.
		\end{aligned}
	\end{equation}
	
	Then we have
	\begin{equation*}
		\begin{aligned}
			\textsf{E}\left|X_t-X_{t_0}\right|^p=&\int_0^{\infty}pu^{p-1}\textsf{P}\left(\left|X_t-X_{t_0}\right|>u\right)du\\
			&=\Delta^p(T)p\left(\int_0^1+\int_1^{\infty}\right)u^{p-1}\textsf{P}\left(\left|X_t-X_{t_0}\right|>u\Delta(T)\right)du\\
			&\leq\Delta^p(T)p\left(2\int_0^1u^{p-1}\exp\left(-\frac{u^2}{2}\right)du+2\int_1^{\infty}u^{p-1}\exp\left(-\frac{u}{2}\right)du\right)\\
			&\leq\Delta^p(T)p\left(2^{\frac{p}{2}}\int_0^1v^{\frac{p}{2}-1}\exp(-v)dv+2^{p+1}\int_{1/2}^{\infty}v^{p-1}\exp(-v)dv\right)\\
			&\leq\Delta^p(T)\left(2^{\frac{p}{2}}p\Gamma\left(\frac{p}{2}\right)+2^{p+1}p\Gamma(p)\right),
		\end{aligned}
	\end{equation*}
	where $\Gamma\left(\cdot\right)=\int_0^{\infty}v^{\cdot-1}\exp(-v)dv$ is the gamma function.
	
	By Stirling's formula, we know for some $0\leq\theta_1(p),\theta_2(p)\leq1$,
	$$2^{\frac{p}{2}}p\Gamma\left(\frac{p}{2}\right)=2\sqrt{\pi}e^{\frac{\theta_1(p/2)}{6p}}p^{\frac{p+1}{2}}\exp\left(-\frac{p}{2}\right),$$
	$$2^{p+1}p\Gamma\left(p\right)=\sqrt{\pi}e^{\frac{\theta_2(p)}{12p}}2^{p+\frac{3}{2}}p^{p+\frac{1}{2}}\exp\left(-p\right).$$
	
	So
	\begin{equation}\label{appsr}
		\begin{aligned}
			\left(\textsf{E}\left|X_t-X_{t_0}\right|^p\right)^{\frac{1}{p}}&\leq\Delta(T)\left(2\sqrt{\pi}e^{\frac{\theta_1(p/2)}{6p}}p^{\frac{p+1}{2}}\exp\left(-\frac{p}{2}\right)+\sqrt{\pi}e^{\frac{\theta_2(p)}{12p}}2^{p+\frac{3}{2}}p^{p+\frac{1}{2}}\exp\left(-p\right)\right)^{1/p}\\
			&\leq\Delta(T)\left(\left(2\sqrt{\pi}e^{\frac{1}{6p}}\right)^{\frac{1}{p}}e^\frac{1-e}{2e}p^{\frac{1}{2}}+2\left(\sqrt{8\pi}e^{\frac{1}{12p}}\right)^{\frac{1}{p}}e^\frac{1-2e}{2e}p\right).\\
		\end{aligned}
	\end{equation}
	
	The same procedure from (\ref{appsta}) to (\ref{appsr}) also applied to $\left(\textsf{E}|X_{t_0}|^p\right)^{\frac{1}{p}}$, so we can also get
	\begin{equation}\label{appsrt}
		\begin{aligned}
			\left(\textsf{E}|X_{t_0}|^p\right)^{\frac{1}{p}}&\leq\Delta(T)\left(2\sqrt{\pi}e^{\frac{\theta_1(p/2)}{6p}}p^{\frac{p+1}{2}}\exp\left(-\frac{p}{2}\right)+\sqrt{\pi}e^{\frac{\theta_2(p)}{12p}}2^{p+\frac{3}{2}}p^{p+\frac{1}{2}}\exp\left(-p\right)\right)^{1/p}\\
			&\leq\Delta(T)\left(\left(2\sqrt{\pi}e^{\frac{1}{6p}}\right)^{\frac{1}{p}}e^\frac{1-e}{2e}p^{\frac{1}{2}}+2\left(\sqrt{8\pi}e^{\frac{1}{12p}}\right)^{\frac{1}{p}}e^\frac{1-2e}{2e}p\right).\\
		\end{aligned}
	\end{equation}
	
	Combing (\ref{t01}), (\ref{t02}), (\ref{t03}) and (\ref{appsrt}), we obtain the moment bound
	\begin{equation}\label{t04}
		\begin{aligned}
			\left(\textsf{E}\sup_{t \in T}|X_t-X_{t_0}|^p\right)^{\frac{1}{p}}\leq C_1\tilde{\gamma}_{\varphi,p}(T,d)+C_2\sqrt{p}\Delta(T)+C_3p\Delta(T).
		\end{aligned}
	\end{equation}
	for some universal constants $C_1,C_2,C_3>0$

	Now using Lemma \ref{DA1}, 
	\begin{equation}\label{t05}
		\begin{aligned}
			\textsf{P}\left(\sup_{t\in T}\left|X_t\right|\geq e\left(c_1\varphi^*(u)+c_2\sqrt{{\varphi^*(u)}}+c_3\right)\right)\leq\exp\left(-\varphi^*(u)\right) \quad  (u\geq\sqrt{2}),
		\end{aligned}
	\end{equation}
	where 
	$$c_1=12e^{\frac{1-2e}{2e}}\Delta(T)\left(\sqrt{8\pi}e^{\frac{1}{12p}}\right)^{\frac{1}{p}},$$
	$$c_2=6e^{\frac{1-e}{2e}}\Delta(T)\left(\sqrt{2\pi}e^{\frac{1}{6p}}\right)^{\frac{1}{p}},$$
	$$c_3=3\left(2+4e^{\frac{1}{2e}}\left(c\sqrt{\frac{2\pi e}{e-1}}\right)^{\frac{1}{p}}\right)\tilde{\gamma}_{\varphi,p}(T,d).$$
	
	To put (\ref{t05}) in another way, for some positive universal constant $C_1, C_2, C_3$,
	\begin{equation}\label{t06}
		\begin{aligned}
			\textsf{P}\left(\sup_{t\in T}\left|X_t\right|\geq C_1\Delta(T)\varphi^*(u)+C_2\Delta(T)\sqrt{\varphi^*(u)}+C_3\tilde{\gamma}_{\varphi,p}(T,d)\right)\leq\exp(-\varphi^*(u)) \quad (u\geq\sqrt{2}).
		\end{aligned}
	\end{equation}
	
	This completes the proof of Theorem \ref{dks}.
\end{proof}

\begin{proof}[Proof of Corollary \ref{signalrlemma}]
	See (\ref{appsrt}) in the proof of Theorem \ref{dks}.
\end{proof}

\section{ Examples and Applications}\label{secE}
This section illustrates a range of applications based on our preceding results.

\subsection{Order 2 Gaussian chaos}
Consider independent standard Gaussian sequences $(g_i)_{i\geq1}$, $(g'_j)_{j\geq1}$ and a given  double sequence sequence $t=(t_{i,j})_{i,j\geq1}$.
\begin{equation}
	X_t=\sum_{i,j\geq1}t_{i,j}g_ig_j
\end{equation}
is defined as a (non-decoupled) order 2 Gaussian chaos.

To lighten the notation, we denote by $tg$ the sequence $\left(\sum_{j\geq1}t_{i,j}g_j\right)_{i\geq1}$, by $\langle\cdot, \cdot\rangle$ the dot product in $\ell^2$, and by $||\cdot||_2$ the corresponding norm. For $t=(t_{i,j})_{i,j\geq 1}$, $||t||_{HS}:=\sum_{i,j\geq1}t_{i,j}^2$ denotes the Hilbert-Schmidt norm and  write
\begin{eqnarray*}
||t||&=&\sup_{ \sum_{j\ge 1}\alpha_{j}^{2}\le 1}(\sum_{ i\ge 1}(\sum_{j\ge 1}\alpha_{j}t_{i,j})^2)^{1/2}\\
     &=& \sup\{\sum_{i\ge 1, j\ge 1}\alpha_j\beta_it_{i,j}: \sum_{j\ge 1}\alpha_{j}^{2}\le 1,  \sum_{i\ge 1}\beta_{i}^{2}\le 1 \}.
\end{eqnarray*}
We define $d_{\infty}(s,t)=||t-s||, $ and 

\begin{equation}
	Y_t^*:=\sum_{i \geq 1}\left(\sum_{j \geq 1} t_{i, j} g_j\right)^2=||t g||_2^2=\langle t g, t g\rangle=\sum_{i \geq 1} \sum_{j, k \geq 1} t_{i, j} t_{i, k} g_j g_k,
\end{equation}

\begin{equation}
	Y_t:=Y_t^*-\textsf{E} Y_t^*=\sum_{i \geq 1} \sum_{j \neq k} t_{i, j} t_{i, k} g_j g_k+\sum_{i \geq 1} \sum_{j \geq 1} t_{i, j}^2\left(g_j^2-1\right),
\end{equation}
which is a Gaussian chaos of order 2 .

\begin{theorem}\label{pchaos}
	For any set $T$ with $0 \in T$, we have, for $p\geq1$,
	\begin{equation}
		\left(\textsf{E} \sup _{t \in T}\left|Y_t\right|^p\right)^{1/p} \leq L \gamma_{2,p}\left(T, d_{\infty}\right)\left(\gamma_{2,p}\left(T, d_{\infty}\right)+\sup _{t \in T}||t||_{H S}\right).
	\end{equation}
\end{theorem}

\begin{remark}
	Gaussian random variables are $\varphi$-sub-Gaussian with $\varphi(x)=x^2/2$. Throught this section, we write $\gamma_{2,p}(T,d)$ for $\gamma_{\varphi,p}$ when $\varphi(x)=x^2/2$.
\end{remark}

\begin{remark}
When $p=1$, Our result is  same as Krahmer, Mendelson and Rauhut \cite{KMR}.\end{remark}

The primary step of the proof of Theorem \ref{pchaos} consists of the following propositions.

\begin{proposition}\label{decouplingl}
	Consider independent standard Gaussian sequences $(g_i)_{i\geq1}$, $(g'_j)_{j\geq1}$ and a given collection $\mathcal{B}$ of double sequence $b=(b_{i,j})_{i,j\geq1}$. We have, for $p\geq1$,
	\begin{equation}\label{decoupling}
		\textsf{E}\sup_{b\in \mathcal{B}}|\sum_{i\neq j}b_{i,j}g_ig_j+\sum_{i\geq1}b_{i,i}(g^2_i-1)|^p\leq2^p\textsf{E}\sup_{b\in \mathcal{B}}|\sum_{i,j\geq1}b_{i,j}g_ig'_j|^p.
	\end{equation}
\end{proposition}

\begin{proof}
	For $b=(b_{i,j})_{i,j\geq1}\in\mathcal{B}$ and $p\geq1$, we have
	\begin{align*}
		\textsf{E}\sup_{b\in \mathcal{B}}|\sum_{i\neq j}b_{i,j}g_ig_j+\sum_{i\geq1}b_{i,i}(g^2_i-1)|^p&=\textsf{E}_g\sup_{b\in \mathcal{B}}|\textsf{E}_{g'}\sum_{i,j}b_{i,j}(g_i+g'_i)(g_j-g'_j)|^p\\
		&\leq\textsf{E}_g\textsf{E}_{g'}\sup_{b\in\mathcal{B}}|\sum_{i,j}b_{i,j}(g_i+g'_i)(g_j-g'_j)|^p\\
		&=\textsf{E}\sup_{b\in\mathcal{B}}|\sum_{i,j}b_{i,j}(g_i+g'_i)(g_j-g'_j)|^p\\
		&=2^p\textsf{E}\sup_{b\in\mathcal{B}}\left|\sum_{i,j}b_{i,j}\frac{g_i+g'_i}{\sqrt{2}}\frac{g_j-g'_j}{\sqrt{2}}\right|^p\\
		&=2^p\textsf{E}\sup_{b\in\mathcal{B}}|\sum_{i,j\geq1}t_{i,j}g_ig'_j|.			
	\end{align*}
	Here, $\textsf{E}_g$ and $\textsf{E}_{g'}$ represents taking expectation in the random variables $g_i$ and $g'_i$, respectively. Jensen's inequality guarantees the inequality above. The last equality above holds since the families $(g_i+g'_i)/\sqrt{2}$ and $(g_i-g'_i)/\sqrt{2}$ are independent sequences of standard Gaussian random variables independent of each other.
\end{proof}

Define
\begin{equation*}
	Z_t=\sum_{i, j, k \geq 1} t_{i, j} t_{i, k} g_j g'_k=\left\langle t g, tg'\right\rangle.
\end{equation*}

\begin{proposition} \label{prop1}
	For $r\geq2$, let $U_r:=\textsf{E}\sup_{t \in T}||t g||_2^r$. Then, for $p\geq1$,
	\begin{equation}\label{prop1.0}
		\left(\textsf{E} \sup _{t \in T}\left|Z_t\right|^p\right)^{1/p} \leq L U_{2p}^{1/2p} \gamma_{2,p}\left(T, d_{\infty}\right) .
	\end{equation}
\end{proposition}

\begin{proof}
	Without loss of generality, it might be assumed that $T$ is finite.
	Consider an admissible sequence $\left(\mathcal{A}_n\right)$ with
	$$
	\sup _{t \in T} \sum_{n \geq \lfloor\log p/\log 2\rfloor} 2^{n / 2} \Delta\left(A_n(t)\right) \leq 2 \gamma_{2,p}\left(T, d_{\infty}\right),
	$$
	where the diameter $\Delta$ is for the distance $d_{\infty}$. For $A \in \mathcal{A}_n$, consider an element $t_{A, n} \in A$, and define as usual a chaining by $\pi_n(t)=t_{A_n(t), n}$. Since $0 \in T$, without loss of generality, we may assume that $\pi_0(t)=0$. We observe that
	\begin{align}\label{Zdecom}
		Z_{\pi_n(t)}-Z_{\pi_{n-1}(t)}=\left\langle\left(\pi_n(t)\right.\right. & \left.\left.-\pi_{n-1}(t)\right) g, \pi_n(t) g^{\prime}\right\rangle \\ \nonumber
		& +\left\langle\pi_{n-1}(t) g,\left(\pi_n(t)-\pi_{n-1}(t)\right) g^{\prime}\right\rangle .
	\end{align}
	
	Recalling that we consider each $t$ as an operator on $\ell^2$, let us denote by $t^*$ its adjoint. Thus,
	\begin{equation}\label{congaussian}
		\left\langle\left(\pi_n(t)-\pi_{n-1}(t)\right) g, \pi_n(t) g^{\prime}\right\rangle=\left\langle g,\left(\pi_n(t)-\pi_{n-1}(t)\right)^* \pi_n(t) g^{\prime}\right\rangle .
	\end{equation}
	Here, $\left(\pi_n(t)-\pi_{n-1}(t)\right)^* \pi_n(t) g^{\prime}$ is the element of $\ell^2$ obtained by applying the operator $\left(\pi_n(t)-\pi_{n-1}(t)\right)^*$ to the vector $\pi_n(t) g^{\prime}$. Let us now consider the r.v.s $W=\sup _{t \in T}||t g||_2$ and $W^{\prime}=\sup _{t \in T}\left|\left|t g'\right|\right|_2$. Then
	$$
	\begin{aligned}
		&\left|\left|\left(\pi_n(t)-\pi_{n-1}(t)\right)^* \pi_n(t) g^{\prime}\right|\right|_2 \leq\left|\left|\left(\pi_n(t)-\pi_{n-1}(t)\right)^*\right|\right|\left||\pi_n(t) g^{\prime}\right||_2 \\
		& \leq \Delta\left(A_{n-1}(t)\right) W^{\prime} .
	\end{aligned}
	$$
	It then follows from (\ref{congaussian}) that, conditionally on $g^{\prime}$, the quantity $\left\langle\left(\pi_n(t)-\pi_{n-1}(t)\right) g, \pi_n(t) g^{\prime}\right\rangle$ is a Gaussian r.v. $G$ with $\left(\textsf{E} G^2\right)^{1 / 2} \leq \Delta\left(A_{n-1}(t)\right) W^{\prime}$. Thus, we obtain that for $u \geq 1$,
	$$
	\textsf{P}\left(\left|\left\langle\left(\pi_n(t)-\pi_{n-1}(t)\right) g, \pi_n(t) g^{\prime}\right\rangle\right| \geq 2^{n / 2} u \Delta\left(A_{n-1}(t)\right) W^{\prime}\right) \leq \exp \left(-u^2 2^n / 2\right).
	$$
	Proceeding similarly for the second term in (\ref{Zdecom}), we get
	$$
	\textsf{P}\left(\left|Z_{\pi_n(t)}-Z_{\pi_{n-1}(t)}\right| \geq 2^{n / 2} u \Delta\left(A_{n-1}(t)\right)\left(W+W^{\prime}\right)\right) \leq 2 \exp \left(-u^2 2^n / 2\right) .
	$$
	
	Using that $Z_{\pi_0(t)}=0$, and proceeding just as in the proof of the generic chaining bound, we obtain that for $u \geq L$,
	\begin{equation}
		\textsf{P}\left(\sup _{t \in T}\left|Z_t\right| \geq Lu\gamma _{2}\left(T, d_{\infty}\right)\left(W+W^{\prime}\right)\right) \leq L \exp \left(-u^2\right).
	\end{equation}
	
	In particular, $R=\sup _{t \in T}\left|Z_t\right| /\left(W+W^{\prime}\right)$ satisfies 
	\begin{equation}
		\textsf{P}\left(R \geq Lu\gamma _{2}\left(T, d_{\infty}\right)\right) \leq L \exp \left(-u^2\right).
	\end{equation}
	As a consequence, according to Lemma \ref{EJPA.5}, we have, for any $q\geq 1$,
	\begin{equation}\label{RLp}
		\left(\textsf{E}{R^q}\right)^{1/q}\leq L\gamma_{2,p}(T,d_{\infty}).
	\end{equation}
	
	Notice that $\sup _{t \in T}\left|Z_t\right|=R\left(W+W^{\prime}\right)$ and $\textsf{E} W^{2p}=E W^{\prime 2p}=U_{2p}$, we know that, by Cauchy-Schwarz inequality and (\ref{RLp}), 
	\begin{align}
		\left(\textsf{E} \sup _{t \in T}\left|Z_t\right|^p\right)^{1/p}&=\left(\textsf{E}R^p(W+W')^p\right)^{1/p}\nonumber\\
		&\leq\left(\textsf{E} R^{2p}\right)^{1/2p}\left(\textsf{E}(W+W')^{2p}\right)^{1/2p}\nonumber\\
		&\leq L\gamma_{2,p}(T,d_{\infty})(2^{2p-1}(\textsf{E}W^{2p}+\textsf{E}W'^{2p}))^{1/2p}\nonumber\\
		&\leq L\gamma_{2,p}(T,d_{\infty})(\textsf{E}W^{2p})^{1/2p}=L\gamma_{2,p}(T,d_{\infty})(U_{2p})^{1/2p}.\nonumber\\
	\end{align}
	
	This finally yields Proposition \ref{prop1}.
\end{proof}

\begin{proposition} \label{prop2}
	We set $V=\sup _{t \in T}||t||_{H S}$. Then for $p\geq1$,
	\begin{equation}
		U_{2p}^{1/2p}\leq L\left(\gamma_{2,p}(T,d_{\infty})+V\right).
	\end{equation}
\end{proposition}

\begin{proof}
	We have $V^2=\sup _{t \in T}||t||_{H S}^2=\sup _{t \in T} \sum_{i, j \geq 1} t_{i, j}^2=\sup _{t \in T} \textsf{E} Y_t^*$. For $t \in$ $T$, we have, for $p\geq1$,
	\begin{equation*}
		||t g||_2^{2p}=(Y_t^*)^{p}=\left(Y_t+E Y_t^*\right)^{p} \leq \left(|Y_t|+V^2\right)^{p}\leq\left(\sup_{t\in T}|Y_t|+V^2\right)^{p}.
	\end{equation*}
	Thus,
	\begin{equation}\label{pfprop2}
		U_{2p} \leq \textsf{E} \left(\sup _{t \in T}\left|Y_t\right|+V^2\right)^{p}\leq2^{p-1}\left(\textsf{E}\sup_{t\in T}|Y_t|^{p}+V^{2p}\right).
	\end{equation}
	By Proposition \ref{decouplingl} and Proposition \ref{prop1},  we have
	\begin{equation}\label{sth}
		\textsf{E} \sup _{t \in T}\left|Y_t\right|^{p} \leq 2^{p}\textsf{E} \sup _{t \in T}\left|Z_t\right|^{p}\leq L^p(\gamma_{2,p}(T,d_{\infty}))^{p} U_{2p}^{1/2}.
	\end{equation}
	Plug (\ref{sth}) into (\ref{pfprop2}), we know that
	$$U_{2p}\leq L^p((\gamma_{2,p}(T,d_{\infty}))^{p} U_{2p}^{1/2}+L^pV^{2p}.$$
	Hence, 
	$$U_{2p}^{1/2}\leq L^p(\gamma_{2,p}(T,d_{\infty}))^{p}+L^p V^{p},$$
	namely,
	$$U_{2p}^{1/2p}\leq L\left(\gamma_{2,p}(T,d_{\infty})+V\right).$$
	
	This proves Proposition \ref{prop2}.
\end{proof}

Now, we are well prepared for the proof of Theorem \ref{pchaos}.
\begin{proof}[Proof of Theorem \ref{pchaos}]
	A combination of Proposition \ref{decouplingl}, Proposition \ref{prop1}, and Proposition \ref{prop2} gives Theorem \ref{pchaos}.
\end{proof}

\subsection{Johnson-Lindenstrauss Lemma}\label{sec-Intro}
This section deals with the Johnson-Lindenstrauss lemma in $\varphi$-sub-Gaussian setting.

\begin{definition}
	A random vector $X$ is called a $\varphi$-sub-Gaussian random vector if all one-dimensional marginals of $X$, i.e., the random variables $\langle X, x\rangle$ for any $x \in \mathbb{R}^n$, are $\varphi$-sub-Gaussian. The $\varphi$-sub-Gaussian norm of $X$ is defined as
	$$
	\tau_\varphi(X):=\sup _{x \in \mathrm{S}^{n-1}}\tau_\varphi\left(\langle X, x\rangle\right), 
	$$
	where $S^{n-1}$ denotes the Euclidean unit sphere in $\mathbb{R}^n$.
\end{definition}

\begin{lemma}\label{KKVV}(Kozachenko and Troshki \cite{KOV})
	For $\xi\in Sub_{\varphi^*}(\Omega)$ with mean zero, $\eta=\xi^2\in Sub_{\psi^*}(\Omega)$ with $\psi(x)=\varphi(\sqrt{x})$.
\end{lemma}

\begin{lemma}\label{finitegamma}
	Assume that $\Delta_2-$condition holds for $\varphi(\cdot)$. When $T$ is finite, for any $p\geq1$, $\gamma_{\varphi,p}(T, d) \leq K_\varphi\Delta(T) \varphi^{*(-1)}(\log\operatorname{card} T)$, where $K_{\varphi}$ denotes a constant with respect to $\varphi$.
\end{lemma}

\begin{proof}[Proof of Lemma \ref{finitegamma}]
	Since $\Delta(A_n(t))=0$ when $2^{2^n}\geq\operatorname{card}(T)$, we know that $2^{n^*}\leq C_1\log(\operatorname{card}(T))$ holds for some $n^*$ and $C_1>1$, and $\Delta(A_{n^*+1}(t))=0$.
	For that, we have for some $M_2\in(0,1)$, 
	\begin{align*}
		\gamma_{\varphi,p}(T, d)&\leq\sup_{t\in T}\sum_{n=0}^{n^*}\varphi^{*(-1)}(2^n)\Delta(A_{n}(t))\nonumber\\
		&\leq\Delta(T)\sum_{n=0}^{n^*}\varphi^{*(-1)}(2^n)\nonumber\\
		&\leq\Delta(T)\sum_{n=0}^{n^*}(1-M_2)^{(n^*-n)}\varphi^{*(-1)}(2^{n^*})\nonumber\\
		&\leq C_1\Delta(T)\varphi^{*(-1)}\left(\operatorname{card}(T)\right)\sum_{n=0}^{n^*}(1-M_2)^{n}\nonumber\\
		&= C_1\frac{1-(1-M_2)^{n^*+1}}{M_2}\Delta(T)\varphi^{*(-1)}\left(\operatorname{card}(T)\right)\nonumber\\
		&\leq \frac{C_1}{M_2}\Delta(T)\varphi^{*(-1)}\left(\operatorname{card}(T)\right).\nonumber
	\end{align*}
\end{proof}
\begin{theorem}\label{mcs}
	Let $X$ be a set of $N$ points in $\mathbb{R}^n$ and $\varepsilon \in(0,1)$. Consider an $m \times n$ matrix $A$ whose rows are independent, mean zero, isotropic, and $\varphi^*$-sub-Gaussian random vectors in $\mathbb{R}^n$.  Rescale $A$ by defining the random projection
		$$
	\mathrm{P}:=\frac{1}{\sqrt{m}} A.
	$$
	Assume that $\varphi^{*(-1)}(\cdot)$ satisfies  
	$$
	m \geqslant C_{\varphi} \varepsilon^{-2} \log(N),
	$$
	where $C_{\varphi}$ is an appropriately large constant with respect to $\varphi$. Then, with a probability of at least 
	$$\exp\left(-\frac{m {\varepsilon}^2}{c \Delta(T) }+C_{\varphi} \log (N)\right),$$
	with $c$ denoting a universal constant, the map $\mathrm{P}$ preserves the distances between all points in $X$ with error $\varepsilon$, that is
	\begin{equation}\label{JL1}\quad(1-\varepsilon)||x-y||_2 \leqslant||P x-P y||_2 \leqslant(1+\varepsilon)||x-y||_2 \quad\end{equation} for all $x, y \in X$.
\end{theorem}

\begin{proof}[Proof of Theorem \ref{mcs}]
	We first rewrite (\ref{JL1}) as
	\begin{equation}\label{JL2}
		1-\varepsilon \leqslant||P z||_2 \leqslant 1+\varepsilon \quad \text { for all } z \in \mathrm{T},
	\end{equation}
	where
	$$
	\mathrm{T}:=\left\{\frac{x-y}{||x-y||_2}: x, y \in X \text { distinct points }\right\} .
	$$
	
	To prove the (\ref{JL2}), it is enough to prove
	\begin{equation}\label{JL3}
		1-\varepsilon \leqslant||P z||_2^2 \leqslant 1+\varepsilon \quad \text { for all } z \in \mathrm{T}.
	\end{equation}
	
	Observe that
	\begin{equation}\label{JL4}
		\begin{aligned}
			\textsf{P}\left(\sup_{z\in T}\left|\frac{1}{m} \sum_{i=1}^m\left\langle X_i, z\right\rangle^2-1\right| \leqslant \varepsilon \right)&=\textsf{P}\left(\sup_{z\in T}\left|\sum_{i=1}^m\left\langle X_i, z\right\rangle^2-m\right| \leqslant m\varepsilon \right)\\
			&=\textsf{P}\left(\sup_{z\in T}\left|\sum_{i=1}^m\left\langle X_i, z\right\rangle^2-m \mathrm{E}\left\langle X_i, z\right\rangle^2\right| \leqslant m\varepsilon \right)\\
			&=\textsf{P}\left(\sup_{z\in T}\left|\sum_{i=1}^m\left(\left\langle X_i, z\right\rangle^2- \mathrm{E}\left\langle X_i, z\right\rangle^2\right)\right| \leqslant m\varepsilon \right),
		\end{aligned}
	\end{equation}
	where the second equality holds because of isotropicity.
	
	By assumption, the random variables $\left\langle X_i, z\right\rangle^2-1$ are independent. For $\left\langle X_i, z\right\rangle \in Sub_{\varphi^*}(\Omega)$, by Lemma \ref{KKVV}, we know that $\left\langle X_i, z\right\rangle^2- \mathrm{E}\left\langle X_i, z\right\rangle^2 \in \mathrm{Sub}_{\psi^*}(\Omega)$, where $\psi(x)=\varphi(\sqrt x)$.
	
	By Theorem \ref{dks}, for $u\geq \sqrt 2$, we have, for $\psi^*$-sub-Gaussian process $Z_t$, 
	$$\textsf{P}\left(\sup _{t \in T}\left|Z_t\right| \geq c (\Delta_\rho(T) u+ \tilde\gamma_{\varphi}(T, \rho))\right) \leq \exp \left(-u\right),$$ 
	for some universal constant $c>0$.
	
	Combining Lemma \ref{finitegamma} and Theorem \ref{eq1}, we know that
	$$\tilde\gamma_{\varphi,p}(T, \rho) \leq K_\varphi\Delta_{\rho}(T) \varphi^{*(-1)}(\log\operatorname{card} T).$$
	Hence, for any $z \in T$, 
	\begin{equation}\label{JL5}
		\begin{aligned}
			\textsf{P}\left\{\sup _{z \in \mathrm{T}}\left|\frac{1}{m}\sum_{i=1}^m\left\langle X_i, z\right\rangle^2-1\right|>\varepsilon\right\}&\leq \exp\left(-\frac{m \varepsilon -c \tilde\gamma_{\varphi}(T, \rho)}{c \Delta(T)}\right)\\
			&\leq \exp\left(-\frac{m \varepsilon-cK_{\varphi} \Delta(T) \varphi^{*(-1)}(\log \operatorname{card} (T))}{c \Delta(T)}\right)\\
			&\leq \exp\left(-\frac{m \varepsilon-c'K_{\varphi} \Delta(T) \varphi^{*(-1)}(\log N)}{c \Delta(T)}\right)\\
            &\leq \exp\left(-\frac{m {\varepsilon}^2-c'K_{\varphi} \Delta(T) \log N}{c \Delta(T)}\right),
		\end{aligned}
	\end{equation}
	where the second-to-last inequality holds because $\operatorname{card} (T)\leqslant N^2$ according to the initial setting of $T$ and the last inequality is based on the properties of $\varphi^{-1}(\cdot)$.
	
	Therefore, for $m \geqslant C_{\varphi} \varepsilon^{-2} \log(N)$ with sufficiently large constant $C_{\varphi}$, The proof is complete. 
	
\end{proof}

\subsection{Convex signal recovery from $\varphi$-sub-Gaussian measurements}
This section develops some problems in signal recovery. We demonstrate a universal error bound for $\varphi$-sub-Gaussian measurements. To achieve the bound, we obtain a lower bound for the minimum conic singular value for a random matrix $\mathbf{\Phi}$ that satisfies certain conditions, which is of independent interest.

Some background material on signal recovery is presented first, and then we give the main theorem.
\begin{definition}
	A set $\mathbf{C}\subseteq\mathbb{R}^n$ is called a cone if $\mathbf{C}=\theta \mathbf{C}$ for all $\theta>0$. For a proper convex function $f:\mathbb{R}^n\to\bar{\mathbb{R}}$, the descent cone $\mathcal{D}(f,\mathbf{x})$ of the function $f$ at a point $\mathbf{x}\in\mathbb{R}^n$ is defined as 
	$$\mathcal{D}(f,\mathbf{x})=\bigcup_{\theta>0}\{\mathbf{u}\in\mathbb{R}^n,\quad f(\mathbf{x+\theta u})\leq f(\mathbf{x})\}.$$
\end{definition}

\begin{remark}
	The descent cone of a convex function is always a convex cone, though not necessarily closed.
\end{remark}

\begin{definition}
	For a $m\times d$ matrix $\mathbf{\Phi}$ and a cone $\mathbf{C}\in\mathbb{R}^n$ not necessarily convex, the minimum singular value of $\mathbf{\Phi}$ with respect to the cone $\mathbf{C}$ is defined as 
	$$\lambda_{min}(\mathbf{\Phi};\mathbf{C})=\inf_{\mathbf{u}\in\mathbf{C}\cap S^{n-1}}||\mathbf{\Phi u}||_2,$$
	with $S^{n-1}$ denoting the Euclidean unit sphere in $\mathbb{R}^n$. Throughout this section, $||\cdot||_2$ denotes the Euclidean $\ell$-2 norm.
\end{definition}

We briefly recall the framework for many convex optimization methods for recovering a structured signal from linear measurements.

For an unknown but structured signal $\mathbf{x}^*\in\mathbb{R}^n$, suppose we have observed a vector $\mathbf{y}\in\mathbb{R}^m$ that consists of $m$ linear measurements $\mathbf{y}=\mathbf{\Phi x^*+e}$. It is assumed that $\mathbf{\Phi}$ is a known $m\times n$ sampling matrix of the form $\mathbf{\Phi}=(\mathbf{\varphi_1},\cdots,\mathbf{\varphi_m})^t$ with $\mathbf{\varphi_i}\in\mathbb{R}^n\,(i=1,\cdots,m)$ independent and identically distributed random vectors, and $\mathbf{e}\in\mathbb{R}^m$ is vector of unknown noises. 
It is expected to reconstruct the unknown $\mathbf{x}^*$ via convex optimization. For a proper convex function $f:\mathbb{R}^n\to\bar{\mathbb{R}}$, the following convex program is usually studied:
\begin{align}\label{signalre}
	minimize \quad f(\mathbf{x}) \quad subject \,\,\,to\quad  ||\mathbf{\Phi x-y}||_2\leq \eta
\end{align}
for a specified bound $\eta$ on the norm of the noise.

The following proposition in  Tropp \cite{TR} provides an error bound for convex recovery.
\begin{proposition}\label{sreb1}(Tropp \cite{TR})
	For any solution $\mathbf{x}_{\eta}$ to the the convex optimization problem described in (\ref{signalre}), we have
	\begin{align}
		||\mathbf{x_{\eta}-x^*}||_2\leq\frac{2\eta}{\lambda_{min}(\mathbf{\Phi},\mathcal{D}(f,\mathbf{x^*}))}.
	\end{align}
\end{proposition}

From the proposition \ref{sreb1}, we know the task for bounding the error of convex signal recovery is to obtain the lower bound of $\lambda_{min}(\mathbf{\Phi},\mathcal{D}(f,x^*))$.
Mendelson \cite{M2015} obtained a lower bound of the minimum conic singular value as a nonnegative empirical process, providing a powerful tool.  Tropp \cite{TR} got the following result. 

\begin{proposition}\label{smallb} (Tropp \cite{TR})
	A cone $\mathbf{C}\subseteq\mathbb{R}^n$ is fixed. Suppose $\mathbf{\varphi_1},\cdots,\mathbf{\varphi_m}$ are independent copied of random vector $\mathbf{\varphi_0}\in\mathbb{R}^n$ and the sampling matrix $\mathbf{\Phi}$ is of the form $\mathbf{\Phi}=(\mathbf{\varphi_1},\cdots,\mathbf{\varphi_m})^t$. Then for any $\xi>0$ and $t>0$, with probability exceeding $1-e^{-t^2/2}$
	\begin{equation}
		\begin{aligned}
			\lambda_{min}(\mathbf{\Phi};\mathbf{C})&=\inf_{u\in\mathbf{C}\cap S^{n-1}}\left(\sum_{i=1}^m|\left\langle\mathbf{\varphi_i,u}\right\rangle|^2\right)^{1/2}\\
			&\geq\xi\sqrt{m}Q_{2\xi}(\mathbf{C}\cap S^{n-1};\mathbf{\varphi_0})-2W_m(\mathbf{C}\cap S^{n-1};\mathbf{\varphi_0})-\xi t
		\end{aligned}
	\end{equation}
	Here, $$Q_{2\xi}(\mathbf{C}\cap S^{n-1};\mathbf{\mathbf{\varphi_0}})=\displaystyle\inf_{u\in\mathbf{C}\cap S^{n-1}}\textsf{P}\{|\langle \mathbf{\varphi_0,u}\rangle|\geq\xi\}\quad\text{for}\quad\xi>0,$$ and 
	$$W_m(\mathbf{C}\cap S^{n-1};\mathbf{\varphi_0})=\textsf{E}\sup_{u\in\mathbf{C}\cap S^{n-1}}\langle \frac{1}{\sqrt{m}}\sum_{i=1}^m\varepsilon_i\mathbf{\varphi_i,u}\rangle,$$
	with $\varepsilon_i\,\,(i=1,\cdots,m)$ independent Rademacher random variables independent from $\mathbf{\varphi_i}\,\,(i=1,\cdots,m)$.
\end{proposition}

With all these preliminaries, we are finally well prepared for the error bound for signal recovery under $\varphi$-sub-Gaussian measurements by Mendelson's Small Ball Method (see Mendelson\cite{M2015} and  Tropp \cite{TR} for more details about this technique). In this section, $T=\{1,\cdots,m\}$ is  a metric space endowed with the distance $d(i,j)=\tau_{\varphi}(\mathbf{\varphi_i,\varphi_j})$ for $1\leq i,j\leq m$. Some assumptions are described as follows. It is assumed that $\mathbf{\varphi_0}$ is a random vector in $\mathbb{R}^n$ satisfying:
\begin{itemize}
	\item $\textsf{E}\mathbf{\varphi_0}=0$; [Centering]
	\item For some constant $\alpha>0$, $\textsf{E}|\langle\mathbf{\varphi_0},\mathbf{u}\rangle|\geq\alpha$ for all $\mathbf{u}\in S^{n-1}$; [Nondegeneracy] 
	\item $\mathbf{\varphi_0}$ is a $\varphi$-sub-Gaussian random vector; [$\varphi$-sub-Gaussian marginals] 
	\item The eccentricity $\mu:= \Delta(T)/\alpha$ of the distribution should be small. [Low eccentricity] 
\end{itemize}

Under these assumptions, the following theorem demonstrates a lower bound for the minimum conic singular value of a random matrix $\mathbf{\Phi}$.

\begin{theorem}\label{minsingle}
	Suppose $\mathbf{\Phi=(\varphi_1\cdots\varphi_m)^t}$ is an $m\times d$ random matrix with $\mathbf{\varphi_i}\,\,(i=1,\cdots,m)$ independent and identically distributed copied of  $\mathbf{\varphi_0}$.  Random vector $\mathbf{\varphi_0}$ satisfies the above-mentioned conditions. For a cone $\mathbf{C}\subset\mathbb{R}^n$ not necessarily convex, with probability exceeding $1-e^{-t^2/2}$,
	$$\lambda_{min}(\mathbf{\Phi};\mathbf{C})\geq C_1\sqrt{m} \alpha\mu^{-2}-C_2\gamma_{\varphi,p}(T, d)+C_3\Delta(T)-\alpha t/3,$$
	with $C_1,C_2,C_3>0$ some universal constants.
\end{theorem}

\begin{proof}
	For a fixed cone $\mathbf{C}\in\mathbb{R}^n$, Proposition \ref{smallb} indicated that with probability at least $1-e^{-t^2/2}$, for all $\xi>0$ and $t>0$,
	\begin{equation}\label{basic}
		\begin{aligned}
			\lambda_{min}(\mathbf{\Phi};\mathbf{C})\geq\xi\sqrt{m}Q_{2\xi}(\mathbf{C}\cap S^{n-1};\mathbf{\varphi_0})-2W_m(\mathbf{C}\cap S^{n-1};\mathbf{\varphi_0})-\xi t
		\end{aligned}
	\end{equation}
	
	By Paley-Zygmund inequality, we know that for any $\mathbf{u}\in\mathbf{C}\cap S^{n-1}$,
	\begin{equation*}
		\textsf{P}\{|\langle\mathbf{\varphi_0},\mathbf{u}\rangle|\geq2\xi\}\geq\frac{\left[(\textsf{E}{|\langle\mathbf{\varphi_0},\mathbf{u}\rangle|}-2\xi)\vee0\right]^2}{{\textsf{E}|\langle\mathbf{\varphi_0},\mathbf{u}\rangle|}^2}.
	\end{equation*}
	
	By $\varphi$-sub-Gaussian marginal condition, setting $p=2$ in Corollary \ref{signalrlemma}, we know that
	\begin{equation*}
		\textsf{E}{|\langle\mathbf{\varphi_0},\mathbf{u}\rangle|}^2\leq C_4\Delta^2(T). 
	\end{equation*}
	for some universal constant $C_4>0$.
	
	With the assumption that $\textsf{E}|\langle\mathbf{\varphi_0},\mathbf{u}\rangle|\geq\alpha$, we know that for $\xi<\alpha/2$,
	\begin{equation}\label{QQ}
		Q_{2\xi}(\mathbf{C}\cap S^{n-1};\mathbf{\varphi_0})=\displaystyle\inf_{u\in\mathbf{C}\cap S^{n-1}}\textsf{P}\{|\langle \mathbf{\varphi_0,u}\rangle|\geq\xi\}\geq C_5\frac{(\alpha-2\xi)^2}{\Delta^2(T)}.
	\end{equation}
	
	The following calculation implies that for random vector $\mathbf{h}=\frac{1}{\sqrt{m}}\sum_{i=1}^ m\varepsilon_i\mathbf{\varphi_i}$, $\langle\mathbf{h,u}\rangle$ is a $\varphi$-sub-Gaussian random variable for each $\mathbf{u}\in\mathbb{R}^n$. For $\lambda>0$,
	\begin{align*}
		\textsf{E}\exp{(\lambda\langle\mathbf{h,u}\rangle)}&=\textsf{E}\exp{\left(\frac{\lambda}{\sqrt{m}}\sum_{i=1}^m\varepsilon_i\langle\mathbf{\varphi_i,u}\rangle\right)}\\
		&=\textsf{E}_X\textsf{E}_{\varepsilon}\exp{\left(\frac{\lambda}{\sqrt{m}}\sum_{i=1}^m\varepsilon_i\langle\mathbf{\varphi_i,u}\rangle\right)}\\
		&=\frac{1}{2}\textsf{E}_X\exp{\left(\frac{\lambda}{\sqrt{m}}\sum_{i=1}^m\langle\mathbf{\varphi_i,u}\rangle\right)}+\frac{1}{2}\textsf{E}_X\exp{\left(-\frac{\lambda}{\sqrt{m}}\sum_{i=1}^m\langle\mathbf{\varphi_i,u}\rangle\right)}\\
		&\leq \frac{1}{2}\exp{\left(\varphi\left(\frac{\lambda}{\sqrt{m}}\tau_{\varphi}(\sum_{i=1}^m\langle\mathbf{\varphi_i,u}\rangle)\right)\right)}+\frac{1}{2}\exp{\left(\varphi\left(\frac{\lambda}{\sqrt{m}}\tau_{\varphi}(-\sum_{i=1}^m\langle\mathbf{\varphi_i,u}\rangle)\right)\right)}\\
		&=\exp{\left(\varphi\left(\frac{\lambda}{\sqrt{m}}\tau_{\varphi}(\sum_{i=1}^m\langle\mathbf{\varphi_i,u}\rangle)\right)\right)}\leq\exp{\left(\varphi\left(\frac{\lambda}{\sqrt{m}}\sum_{i=1}^m\tau_{\varphi}(\langle\mathbf{\varphi_i,u}\rangle)\right)\right)},
	\end{align*}
	with the last two inequalities using the basic property of the norm. The subscripts $\varepsilon$ and $X$ of expectation are intended to remind us about the sources of randomness used in taking these expectations.
	
	We can also know from the calculation above that
	$$\tau_{\varphi}(\langle\mathbf{h,u}\rangle)=\tau_{\varphi}\left(\frac{1}{\sqrt{m}}\sum_{i=1}^m\varepsilon_i\langle\mathbf{\varphi_i,u}\rangle\right)\leq\frac{1}{\sqrt{m}}\sum_{i=1}^m\tau_{\varphi}\left(\langle\mathbf{\varphi_i,u}\rangle\right).$$
	
	Then we get for $\mathbf{u,v}\in\mathbb{R}^n$,
	\begin{align*}
		\tau_{\varphi}(\langle\mathbf{h,u-v}\rangle)\leq\frac{1}{\sqrt{m}}\sum_{i=1}^m\tau_{\varphi}\left(\langle\mathbf{\varphi_i,u-v}\rangle\right).
	\end{align*}
	
	According to Theorem \ref{dks}, we know that 
	\begin{align}\label{Wm}
		W_m(\mathbf{C}\cap S^{n-1};\mathbf{\varphi_0})=\textsf{E}\sup_{u\in\mathbf{C}\cap S^{n-1}}\langle\mathbf{h,u}\rangle\leq  C_6\gamma_{\varphi,p}(T, d)+C_7\Delta(T),
	\end{align}
	for some constant $C_6, C_7>0$.
	
	Combing (\ref{basic}), (\ref{QQ}) and (\ref{Wm}), we obtain that for $\xi<\alpha/2$, with probability at least $1-e^{-t^2/2}$,
	\begin{equation*}
		\lambda_{min}(\mathbf{\Phi};\mathbf{C})\geq C_8\xi\sqrt{m} \frac{(a-2\xi)^2}{\Delta^2(T)}-C_9\gamma_{\varphi,p}(T, d)+C_{10}\Delta(T)-\xi t
	\end{equation*}
	for some for some constant $C_8, C_9, C_{10}>0$. Select $\xi=\alpha/3$, and then we get ,with probability at least $1-e^{-t^2/2}$,
	\begin{equation}
		\lambda_{min}(\mathbf{\Phi};\mathbf{C})\geq C_{11}\sqrt{m} \frac{\alpha^3}{\Delta^2(T)}-C_{12}\gamma_{\varphi,p}(T, d)+C_{13}\Delta(T)-\alpha t/3.
	\end{equation}
	Simplify this result by eccentricity $\mu=\Delta(T)/\alpha$, and the desired theorem is proved.
\end{proof}

With Theorem \ref{minsingle} in hand, we immediately have an error bound for signal recovery from $\varphi$-sub-Gaussian measurements.

\begin{theorem}\label{mcsr}
	Suppose $\mathbf{x}^*$ is a signal in $\mathbb{R}^n$, and $\mathbf{\Phi=(\varphi_1\cdots\varphi_m)^t}$ is an $m\times d$ random matrix with $\mathbf{\varphi_i}\,\,(i=1,\cdots,m)$ independent and identically distributed copied of $\mathbf{\varphi_0}$.  Random vector $\mathbf{\varphi_0}$ satisfies the four conditions, namely centering, nondegeneracy, $\varphi$-sub-Gaussian marginals, and low eccentricity. $\mathbf{y=\Phi x^*+e}$ is a vector of measurements in $\mathbb{R}^m$, where $\mathbf{e}$ is noise.  Suppose $||\mathbf{e}||\leq\eta$ and $\mathbf{x_\eta}$ is any solution to the problem (\ref{signalre}), then with probability more than $1-e^{-t^2/2}$, we have
	\begin{equation*}
		||\mathbf{x_\eta-x^*}||_2\leq\frac{2\eta}{\left(C\sqrt{m} \frac{\alpha^3}{\Delta^2(T)}-C'\gamma_{\varphi,p}(T, d)+C''\Delta(T)-\alpha t/3\right)\vee0},
	\end{equation*}
	with $C,C',C''$ some universal constants.
\end{theorem}

\begin{proof}
	A combination of Proposition \ref{sreb1} and Theorem \ref{minsingle} completes the proof.
\end{proof}

\section{Discussions}
This work presents the first generic chaining argument for $\varphi$-sub-Gaussian processes as a generalization of Gaussian processes, and it has been demonstrated that the results obtained through the chaining method are superior to the Dudley bound in this space. We also attempt to extend the recent advances building on Talagrand's functional in this paper, for instance, the truncated version of the $\gamma$-functional attributed to Krahmer, Mendelson and Rauhut \cite{KMR}. Moreover, we attempt to verify the growth condition and explore partition schemes for the distribution-dependent Talagrand-type $\gamma$-functional. On the other hand, studying the lower bound for the expectation of the supremum of $\varphi$-sub-Gaussian processes is interesting. Recently, the so-called Bernoulli conjecture has been solved in  Bednorz and  Latała \cite{bernoulli}, and later,  Bednorz and Martynek \cite{cbernoulli} obtained a more general result.  These authors established the decomposition theorems for the lower bound of the Bernoulli process,  empirical processes, and so on.  After these great works,  it will be fascinating to study the decomposition theorem for the canonical process $\sum_{i\ge 1}t_i\xi_i$, where $\{\xi_i\}_{i\ge 1}$ is a sequence of $\varphi$-sub-Gaussian random variables.

\section *{Acknowledgments}

This work was supported by the National Natural Science Foundation of China (No. 12071257 and No. 11971267 );   National Key R$\&$D Program of China (No. 2018YFA0703900 and No. 2022YFA1006104); Shandong Provincial Natural Science Foundation (No. ZR2019ZD41).

\vskip3mm

\bibliographystyle{amsplain}

\end{document}